\documentclass[a4paper,11pt]{article}
\usepackage[utf8]{inputenc}
\usepackage[T1]{fontenc}
\usepackage{fouriernc}
\usepackage{amsmath}
\usepackage{amssymb}
\usepackage{amsthm}
\usepackage{enumitem}
\usepackage{graphicx}
\usepackage{xcolor}
\usepackage{tikz}
\usepackage{comment}
\usepackage{mathtools}
\usetikzlibrary{arrows,positioning}
\usetikzlibrary{calc}
\setlist{nolistsep}
\usepackage[margin=20mm,includehead,includefoot]{geometry}

\usepackage{fancyhdr}
\fancypagestyle{plain}{
\fancyhf{}
\fancyfoot [C] {\textbf{\thepage}}

}

\newcommand{\I}{\mathcal{I}}
\newcommand{\SG}{\mathsf{SG}}
\newcommand{\SC}{\mathsf{SC}}

\definecolor {cornellrot} {RGB} {179,27,27}
\usepackage[colorlinks=true,linkcolor=cornellrot,citecolor=cornellrot,urlcolor=cornellrot]{hyperref}

\newtheoremstyle{graz}{10pt}{5pt}{\normalfont\it}{}{\normalfont\bfseries}{}{6pt}{\textbf{\thmname{#1}\thmnumber{ #2}\thmnote{ (#3)}}}
\theoremstyle{graz}

\newtheorem{rem}{Remark}[section]
\newtheorem{lem}[rem]{Lemma}
\newtheorem{remark}[rem]{Remark}
\newtheorem{prop}[rem]{Proposition}
\newtheorem{defn}[rem]{Definition}
\newtheorem{thm}[rem]{Theorem}

\AtBeginDocument{
  \DeclareSymbolFont{AMSb}{U}{msb}{m}{n}
  \DeclareSymbolFontAlphabet{\mathbb}{AMSb}}

\begin{document}
\allowdisplaybreaks

\title{Internal DLA on Sierpinski gasket graphs}
\author{Joe P.\@ Chen, Wilfried Huss, Ecaterina Sava-Huss, Alexander Teplyaev}
\date{\today}
\maketitle

\begin{abstract}
Internal diffusion-limited aggregation (IDLA) is a stochastic growth model on a graph $G$ which describes the formation of a random set of vertices growing from the origin (some fixed vertex) of $G$. Particles start at the origin and perform simple random walks; each particle moves until it lands on a site which was not previously visited by other particles. This random set of occupied sites in $G$ is called the IDLA cluster.
In this paper we consider IDLA on Sierpinski gasket graphs, and show that the IDLA cluster fills balls (in the graph metric) with probability 1.
\end{abstract}


\section{Introduction}\label{sec:intro}

The \emph{internal diffusion limited aggregation} model (shortly \emph{IDLA} or \emph{internal DLA}) is a stochastic growth model introduced by 
\textsc{Diaconis and Fulton} in \cite{diaconis_fulton_1991} as an \emph{internal counterpart} of the \emph{external DLA}, model introduced in physics \cite{dla}. For a survey on external and internal DLA, see \cite{survey-dla}. To formally define the process, let $G$ be an infinite connected graph with a distinguished vertex $o$ which will be called the origin.
Then IDLA on $G$ is defined as follows. For $i=1,2,\ldots$, let $\left(X^i\left(t\right)\right)_{t\geq 0}$
be a sequence of iid simple random walks on $G$ starting at $o$, where $X^i(t)$ represents the random position of the $i$th random walk
at time $t$. The IDLA cluster is built  up one site at a time, by letting the
$i$th particle walk until it exits the set of sites already occupied by the previous $i-1$
particles. Denote by $\mathcal{I}(i)$ the IDLA cluster with $i$ particles.
Set $\mathcal{I}(0)=\{o\}$, and for $i\geq 1$ define the sequence of stopping times
$(\sigma^i)_{i\geq 0}$, with $\sigma^0=o$ and
\begin{equation*}
\sigma^i=\inf\big\{t>0:X^i(t)\notin \mathcal{I}(i-1)\big\}.
\end{equation*}
The IDLA cluster with $i$ particles is defined inductively as
\begin{equation}\label{eq:idla-cluster}
\mathcal{I}(i)=\mathcal{I}(i-1)\cup \{X^i(\sigma^i)\}.
\end{equation}
For $i$ large, we are interested in the shape of the IDLA cluster $\mathcal{I}(i)$
after the $i$th particle stops. Does the random set $\mathcal{I}(i)$ exhibit a regular shape once $i$ is large enough? On $\mathbb{Z}^d$ {\sc Lawler, Bramson and Griffeath} \cite{lawler_bramson_griffeath} were the first to identify the limit shape as an Euclidean ball. If instead of simple random walks on $\mathbb{Z}^d$ one takes drifted random walks, then the limiting shape is shown to be a true heat ball in \cite{lucas-idla-drift};
the proof is based on an unfair divisible sandpile model. On other state spaces, there are several results concerning the IDLA limit shape: on discrete groups with exponential growth \cite{blachere-brofferio-idla}, on non-amenable graphs  \cite{huss-nonamenable}, on supercritical percolation clusters on $\mathbb{Z}^d$ \cite{shellef-percolation,idla-outerbd-dum-cop}, and on comb lattices \cite{idla-comb-huss-sava, asselah-comb}.
For a survey of these results see \cite{survey-dla}.

In this work we investigate IDLA on the pre-fractal Sierpinski gasket graph, specifically the doubly infinite Sierpinski gasket graph $\SG$, shown in Figure \ref{fig:gasket}. 
To construct $\SG$, we consider in $\mathbb{R}^2$ the sets
\begin{equation*}
V_0=\{(0,0), (1,0), (1/2,\sqrt{3}/2)\}
\end{equation*}
and 
\begin{equation*}
E_0=\left\{\big((0,0),(1,0)\big),\big((0,0),(1/2,\sqrt{3}/2)\big),\big((1,0),(1/2,\sqrt{3}/2)\big)\right\}.
\end{equation*}
Now recursively define $(V_1,E_1), (V_2,E_2),\ldots$ by
\begin{equation*}
V_{n+1}=V_n\cup\left\{\big(2^n,0\big)+V_n\right\}\bigcup \left\{\left(2^{n-1},2^{n-1}\sqrt{3}\right)+V_n\right\}
\end{equation*}
and 
\begin{equation*}
E_{n+1}=E_n\cup\left\{\left(2^n,0\right)+E_n\right\}\bigcup \left\{\left(2^{n-1},2^{n-1}\sqrt{3}\right)+E_n\right\},
\end{equation*}
where $(x,y)+S:=\{(x,y)+s:s\in S\}$. Let $V_{\infty}=\bigcup_{n=0}^{\infty}V_n$, $E_{\infty}=\bigcup_{n=0}^{\infty}E_n$, $V=V_{\infty}\cup \{-V_{\infty}\}$ and $E=E_{\infty}\cup \{-E_{\infty}\}$.
Then the doubly infinite Sierpinski gasket graph $\SG$
is the graph with vertex set $V$ and edge set $E$.
Set the origin $o=(0,0)$.

\begin{figure}
	\centering
	\input{gasket.tex}
	\caption{\label{fig:gasket} Doubly-infinite Sierpinski gasket graph.}
\end{figure}

Our main result is the following spherical shape theorem for the IDLA cluster on $\SG$, consisting of random walks launched successively from $o$.
Denote by $B_o(n)$ the ball of radius $n$ and center $o$ in the graph distance of $\SG$, and by $b_n:=\mid B_o(n) \mid$ its cardinality.
\begin{thm}[Shape theorem for internal DLA]\label{thm:idla-sierpinski-gasket}
On $\SG$, the IDLA cluster of $b_n$ particles occupies a set of sites close to
a ball of radius $n$. That is, for all $\epsilon >0 $, we have with probability 1
\begin{equation*}
B_o\left(n(1-\epsilon)\right)\subset \mathcal{I}(b_n)\subset B_o\left(n(1+\epsilon)\right), \text{ for all n sufficiently large}.
\end{equation*}

\end{thm}

Our proof of Theorem \ref{thm:idla-sierpinski-gasket} combines relevant arguments from previous work on IDLA \cite{lawler_bramson_griffeath, lawler_1995,idla-outerbd-dum-cop} and information about the geometry and the potential theory on $SG$ \cite{BP88,Bar98,KigamiBook}. A key argument used in Theorem \ref{thm:idla-sierpinski-gasket} is the fact that $\SG$ is a finitely ramified fractal, and this makes the analysis of random walks \cite{BP88,Bar98} and quantities related to random walks, such as Green functions and harmonic functions \cite{KigamiBook,StrichartzBook}, easier.
Let us remark that it is possible to prove Theorem \ref{thm:idla-sierpinski-gasket} on the one-sided infinite Sierpinski gasket graph (where the origin $o$ has degree $2$) upon appropriate modification of the proofs.

The rest of the paper is organized as follows. In Section \ref{sec:prelim} we give the basic notions on random walks and related quantities such as Green function and harmonic function. In Section \ref{sec:prelim-sg} we recall some well known properties and results on Sierpinski gasket graphs and random walks on them, which will be subsequently used in the proof of the limit shape for the IDLA cluster. Then in Section \ref{sec:idla-inner-bound} and in Section \ref{sec:idla-outer-bd}
we prove, respectively, the inner bound and the outer bound of Theorem \ref{thm:idla-sierpinski-gasket}. We conclude with some remarks and open questions.

\section{Preliminaries}\label{sec:prelim}
\subsection{Random walks}\label{sec:rw}

Let $G$ be any infinite, locally finite, connected graph. We write $G$ also for the set of its vertices and for $x,y\in G$ we write 
$x \sim y$ if $(x,y)$ is an edge in $G$. We denote by $d(x,y)$ the natural graph distance in $G$,
i.e. the length of the shortest path between $x$ and $y$ in $G$. 
Given a subset $A\subset G$, we define its boundary $\partial A =\{ y\in G\setminus A: y \sim x~\text{for some}~x\in A\}$.
For $x\in G$ and $n\geq 0$,
denote by $B_x(n)=\{y\in G: d(x,y)< n \}$ the ball of center $x$ and radius $n$ in $G$, and by
$\partial B_x(n)=\left\{y\in B_x(n)^c: y\sim x \text{ for some }x\in B_x(n)\right\}$ the boundary of $B_x(n)$. 
Denote by $\deg(x)$ the degree of $x$ in $G$, that is, the number of neighbors of $x$.  We would like to point out that we use the vertex degree as volume measure.

The (discrete-time) simple random walk (SRW) $\big(X(t)\big)_{t\geq 0}$ on $G$ is the (time homogeneous) 
Markov chain with one-step transition probabilities given by
\begin{equation*}
p(x,y):=\mathbb{P}[X(t+1)=y\mid X(t)=x]=\frac{1}{\deg(x)}
\end{equation*}
if $y\sim x$, and $0$ otherwise. 
The walk $X(t)$ is reversible with respect to $\deg$, since
\begin{equation*}
\deg(x) \cdot p(x,y) =\deg(y) \cdot p(y,x).
\end{equation*}
We denote by $\mathbb{P}_x$ and $\mathbb{E}_x$ the probability law and the expectation
of the random walk $X(t)$ starting at $x\in G$, and omit the subscript if the random walk starts at the origin $o\in G$ (some fixed vertex to be chosen later).
The $t$-step transition probabilities are then defined as
\begin{equation*}
p_t(x,y)=\mathbb{P}[X(t)=y\mid X(0)=x]
\end{equation*} 
and we have that
\begin{equation*}
p_t(x,y)= \frac{\deg(y)}{\deg(x)}\cdot p_t(y,x).
\end{equation*}

\paragraph*{Green function.} The Green function $g$ is defined by
\begin{equation*}
g(x,y)=\mathbb{E}_x\Big[\sum_{t=0}^{\infty}\mathbf{1}_{\{X(t)=y\}}\Big]=\sum_{t=0}^{\infty}p_t(x,y)
\end{equation*}
and represents the expected number of visits to $y$ of the random walk $X(t)$ started at $x$.
If $X(t)$ is a recurrent random walk, then $g$ is not defined since every vertex is visited infinitely many times.
A quantity of interest in the context of aggregation models is the \emph{stopped Green function} $g_n$ upon exiting 
a ball in the graph. For some vertex $x\in G$, if 
\begin{equation}\label{eq:exit-time}
\tau_n(x)=\inf\{t: X(t)\notin B_x(n)\}
\end{equation}
is the first time when the random walk $X(t)$ exits $B_x(n)$, then
\begin{equation}\label{eq:stopped-green}
g_n(x,y)=\mathbb{E}_x\Big[\sum_{t=0}^{\tau_n(o)-1}\mathbf{1}_{\{X(t)=y\}}\Big]
\end{equation}
and represents the expected number of visits to $y$ before 
time $\tau_n(o)$, of the random walk $X(t)$ starting at $X(0)=x$. We write $\tau_n:=\tau_n(o)$ if there is no risk for confusion.
For $y\in G$ define
\begin{equation}
\tau_y=\inf\{t:X(t)=y\}
\end{equation}
to be the first time the random walk $X(t)$ visits $y\in G$. 
By standard Markov chain theory, we have that
\begin{equation}\label{eq:st_gr_fc}
\mathbb{P}_y[\tau_x<\tau_n]=\dfrac{g_n(y,x)}{g_n(x,x)} \quad \text{and} \quad \mathbb{P}_o[\tau_x<\tau_n]=\frac{g_n(o,x)}{g_n(x,x)},
\end{equation}

For a function $h:G\to \mathbb{R}$, the \emph{(probabilistic) graph Laplacian} of $h$ is defined as
\begin{equation*}
 \Delta h(x)=\frac{1}{\deg(x)}\sum_{y\sim x} h(y)-h(x).
\end{equation*}
We say that $h$ is \emph{harmonic} on $S\subset G$ if $\Delta h=0$ (that is, the discrete mean value property holds) on $S$.

\begin{defn}\label{def:ehi}
We say that the graph $G$ satisfies an \emph{elliptic Harnack inequality (EHI)} if there exists a positive constant $C$ such that for all $x\in G$, $n>0$, and functions $h\geq 0$ which are harmonic on $B_x(2n)$,
\begin{equation}\label{eq:EHI}\tag{EHI}
\sup_{y\in B_x(n)} h(y) \leq C \inf_{y\in B_x(n)}h(y).
\end{equation}
\end{defn}

\begin{defn}\label{def:p0}
A weighted graph $(G,P)$ satisfies the condition ($p_0$) if there exists $p_0 >0$ such that
\begin{equation}\label{eq:p0}\tag{$p_0$}
\frac{p(x,y)}{\sum_{y}p(x,y)} \geq p_0 \text{ for all } x\sim y,
\end{equation}
where $P$ is the transition matrix of the simple random walk $(X(t))$ on $G$.
\end{defn}
The condition ($p_0$) can be viewed as a lower ellipticity bound on the generator (Laplacian) of the random walk.
For more information on analysis on weighted graphs, see also \cite{keller-lenz}.
We shall also use the connection between random walks and electrical networks. For a function $f:G\rightarrow\mathbb{R}$ define its \emph{energy} by 
\begin{equation*}
\mathcal{E}(f)=\frac{1}{2}\sum_{x,y\in G,x\sim y}\left(f(x)-f(y)\right)^2p(x,y)\deg(x),
\end{equation*}
which represents the energy dissipation in the network $G$ associated with the potential $f$.

\begin{defn}
The effective resistance between two (disjoint) subsets $A, B \subset G$ is defined as
\begin{align}
R_{\rm eff}(A,B) = \left[\inf\left\{\mathcal{E}(f) ~|~ f: G\to\mathbb{R},~ f|_A=1,~ f|_B=0 \right\}\right]^{-1},
\end{align}
with the convention that $\inf \emptyset= \infty$.
\end{defn}
We write $R_{\rm eff}(x,y)$ for $R_{\rm eff}\left(\{x\},\{y\}\right)$. Moreover, $R_{\rm eff}(x,y)$ defines a metric on $G$.

\subsection{Sierpinski gasket graphs} 
\label{sec:prelim-sg}

For the rest of the paper, the state space is the doubly-infinite Sierpinski gasket graph denoted by $\SG$; see Figure \ref{fig:gasket}. Let $(X(t))_{t\geq 0}$ be a simple random walk on $\SG$ starting at $o$, which is recurrent. Actually, it is  strongly recurrent both in the sense of \cite[Definition 1.2]{barlow-coulhon-kumagai-hk-2005}. We recall below some known facts about the growth of $\SG$,
the behavior of random walks and Green functions on $\SG$.

\textbf{Notation: }For two sequences $a_n,b_n$ of real numbers, we write $a_n \asymp b_n$, if there exist a constant $C\geq 1$ such that for all $n\in\mathbb{N}$
\begin{equation*}
\frac{1}{C} b_n \leq a_n \leq C b_n. 
\end{equation*}

On $\SG$ there are three main quantities of interest: the spectral dimension $d_s$, the walk dimension $\beta$ (sometimes denoted also $d_w$) and fractal dimension $\alpha$ (called sometimes also uniform volume growth and denoted $d_f$). Throughout this paper we shall use $\alpha$ and $\beta$ for the two quantities mentioned above. They are given by
\begin{equation}\label{eq:alpha-beta}
d_s=2\frac{\log 3}{\log 5}, \quad \beta=\frac{\log 5}{\log 2}\approx 2.32 , \quad \alpha=\frac{\log 3}{\log 2}\approx 1.56,
\end{equation}
and $\beta-\alpha \approx 0.76$. 
\begin{prop}The following holds on $\SG$:
\begin{enumerate}
\item \textbf{Uniform volume growth:} for every $x\in \SG$ and $n\in\mathbb{N}$, the balls $B_x(n)$ around $x$ of radius $n$ have growth of order $\alpha$:
\begin{equation}\label{eq:vg}\tag{VG}
\mid B_x(n)\mid \asymp n^{\alpha},
\end{equation}
\item \textbf{Elliptic Harnack inequality (EHI):} $\SG$ satisfies $(EHI)$; see \cite[Corollary 2.1.8 \& Proposition 3.2.7]{KigamiBook} for a proof.
\end{enumerate}
\end{prop}
The uniform volume growth, called also \emph{Ahlfors regularity condition} 
can be easily deduced on $\SG$; for more details see \cite{Bar98,Barlow-heat-kernels-03}.
Moreover, equation \eqref{eq:vg} implies also the \emph{volume doubling condition}: there exist $c>0$ such that
\begin{equation}\label{eq:vd}\tag{VD}
\mid B_{x}(2n)\mid \leq c \mid B_{x}(n)\mid, \quad \text{ for all } x\in \SG, n \geq 1.
\end{equation}


\begin{rem}
If $P$ is the transition matrix of the simple random walk $X(t)$ on $\SG$, the $(p_0)$ condition is satisfied with $p_0=\frac{1}{4}$.
\end{rem}

The next result is well-known in the study of random walks on fractal graphs, see \emph{e.g.\@} \cite[Proposition 8.11]{Bar98} or \cite[Corollary 2.3]{Barlow-heat-kernels-03}. 

\begin{lem}
\label{lem:exit-time}
Simple random walk on SG is subdiffusive or sub-Gaussian: the expected exit time from balls in $\SG$ has order $\beta>2$, that is for every $x\in \SG$ and radius $n\in\mathbb{N}$
\begin{equation}\label{eq:ebeta}\tag{$E_{\beta}$}
\mathbb{E}_x[\tau_n(x)]  \asymp n^{\beta},
\end{equation}
where $\tau_n(x)$ is defined as in \eqref{eq:exit-time}. 
\end{lem}
In the proof of Theorem \ref{thm:idla-sierpinski-gasket} we will need the expected exit time $\tau_n(o)$ from balls $B_o(n)$ when starting the random walk at an arbitrary point $x$ inside $B_o(n)$. The following upper bound is easily deduced from existing results.

\begin{lem}[Uniform upper bound for the exit time]\label{lem:exit-time-upper-bound}
There exists $C>0$ such that for every $n\geq 1$, $x\in \SG$ and $y\in B_x(n)$, we have
$ \mathbb{E}_y[\tau_n(x)]\leq C n^{\beta}$.
\end{lem}
\begin{proof} 
Since all the conditions from \cite[Proposition 3.4]{barlow-coulhon-kumagai-hk-2005}
are satisfied for $\SG$ with the function $\eta(n)=n^{\beta}$,
the bound follows immediately.
\end{proof}

A corresponding lower bound holds, provided that $x$ is not too close to the boundary of $B_o(n)$.

\begin{lem}[Lower bound for the exit time]\label{lem:exit-time-lower-bound}
For every $\epsilon \in (0,1)$, there exists $c=c(\epsilon)>0$ such that for every $n\geq 1$ and every $x\in B_o(n(1-\epsilon))$, we have
$
\mathbb{E}_x[\tau_n(o)] \geq c n^{\beta}.
$
\end{lem}
\begin{proof}
If $x\in B_o(n(1-\epsilon))$, then $B_x(\epsilon n) \subset B_o(n)$. Since a random walk started at $x$ must exit $B_x(\epsilon n)$ before leaving $B_o(n)$, it follows that
\[
\mathbb{E}_x[\tau_n(o)] \geq \mathbb{E}_x[\tau_{\epsilon n}(x)] \geq C^{-1} \epsilon^\beta n^\beta,
\]
for a constant $C\geq 1$ independent of $n$ and $\epsilon$, where we used the lower bound in \eqref{eq:ebeta}.
\end{proof}

\begin{remark}
The function $h(x)=\mathbb{E}_x[\tau_n]$ solves the Dirichlet problem
\[
\left\{\begin{array}{ll} \Delta h= 1 & \text{in  } B_o(n), \\  h=0 & \text{on } \partial B_0(n).\end{array}\right.
\]
On $SG$ an exact expression of $h$ can be obtained when $n=2^k$ for $k \in \mathbb{N}$. Indeed, in the proof of Lemma \ref{lem:mean-val-gr-fc} below, we will encounter a related Dirichlet problem with $\Delta h=1$ replaced by $\Delta h=1- |B_o(n)|\delta_o$, whose solution is fully addressed in \cite{sandpile-sg-huss-sava}. We believe it is possible to find sharper estimates of $h$ for all radii $n$ using harmonic splines \cite{splines}, see also the recent work \cite{GKQS14}. In any case, the estimates contained in Lemmas \ref{lem:exit-time-upper-bound} and \ref{lem:exit-time-lower-bound} will suffice for the purposes of this paper.
\end{remark}

\begin{lem}[Upper bound for the stopped Green function]\label{lem:up_bd_gn}
There exists $C>0$ such that for every $n\geq 1$ and every $x\in B_o(n)$, 
\begin{equation*}
g_n(x,x)\leq  C n^{\beta-\alpha}\quad \text{and} \quad g_n(o,x)\leq C n^{\beta-\alpha}
\end{equation*}
\end{lem}
\begin{proof}
The graph $\SG$ is a strongly recurrent graph as defined in 
\cite[Definition 1.2]{barlow-coulhon-kumagai-hk-2005}
and has volume growth of exponent $\alpha<\beta$. Therefore it satisfies condition 
$(VG(\beta-))$ from \cite[Theorem 1.3]{barlow-coulhon-kumagai-hk-2005}. Since $|B_x(d(x,y))|\asymp d(x,y)^{\alpha}$, \cite[Theorem 1.3]{barlow-coulhon-kumagai-hk-2005} gives
\begin{equation*}
R_{\rm eff}(x,y)\asymp \dfrac{d(x,y)^{\beta}}{|B_x(d(x,y))|}\asymp d(x,y)^{\beta-\alpha},
\end{equation*}
where $R_{\rm eff}(x,y)$ represents the effective resistance between $x$ and $y$. In view of the connection between random walks and electrical networks as in the proof of \cite[Proposition 3.4]{barlow-coulhon-kumagai-hk-2005}, together with \cite[Theorem 1.3]{barlow-coulhon-kumagai-hk-2005}) we obtain
\begin{equation*}
g_n(x,x)=R_{\rm eff}(x,B_o(n)^c)\leq R_{\rm eff}(x,y)\asymp d(x,y)^{\beta-\alpha}, \quad \text{ for all } x\in SG, y\in B_o(n)^c.
\end{equation*}
The above inequality together with 
$d(x,y)< 2n$  yield $g_n(x,x)\leq C n^{\beta-\alpha}$ for some constant $C>0$, which proves the first part of the claim.
Equation \eqref{eq:st_gr_fc} implies that $g_n(o,x)\leq g_n(x,x)$, and this completes the proof.
\end{proof}

\begin{lem}
The stopped Green function $g_n(o,o)$ has growth of order $\beta-\alpha$:
\begin{equation*}
g_n(o,o)\asymp n^{\beta-\alpha}
\end{equation*}
\end{lem}
\begin{proof}
The upper bound $Cn^{\beta-\alpha}$ follows from Lemma \ref{lem:up_bd_gn} by taking $x=o$.
For the lower bound, from Equation \eqref{eq:ebeta}, we have 
\begin{align*}
\frac{1}{C} n^{\beta}\leq \mathbb{E}_o[\tau_n]&=\sum_{x\in B_o(n)}g_n(o,x)=\sum_{x\in B_o(n)}g_n(x,o)=\sum_{x\in B_o(n)}\mathbb{P}_x[\tau_o<\tau_n]g_n(o,o)\\
&\leq 
\sum_{x\in B_o(n)}g_n(o,o)=|B_o(n)|g_n(o,o)\leq C'n^{\alpha}g_n(o,o),
\end{align*}
which gives
\begin{equation*}
\frac{1}{C\cdot C'}n^{\beta-\alpha}\leq g_n(o,o)\leq C \cdot C' n^{\beta -\alpha},
\end{equation*}
and this finishes the proof.
\end{proof}

\section{IDLA on Sierpinski gasket graphs}
\label{sec:IDLAProof}

In this section we prove the shape result in Theorem \ref{thm:idla-sierpinski-gasket}.
This will be done in two parts: in Section \ref{sec:idla-inner-bound} we prove the inner bound $B_o\left(n(1-\epsilon)\right)\subset \mathcal{I}(b_n)$, and in Section \ref{sec:idla-outer-bd} based on the inner bound, we prove the  outer bound $\mathcal{I}(b_n)\subset B_o\left(n(1+\epsilon)\right)$.
Many of the existing proofs of the limiting shape for the IDLA cluster use the elegant idea of \cite{lawler_bramson_griffeath}
for internal DLA on $\mathbb{Z}^d$. Of course, on other state spaces, different estimates for Green functions, expected exit times from balls
are needed, depending very much on the geometry of the underlying graph.

\subsection{The inner bound for the IDLA cluster}
\label{sec:idla-inner-bound}

The proof of the inner bound in Theorem \ref{thm:idla-sierpinski-gasket},
is based on understanding the divisible sandpile model and its shape on $\SG$. The divisible sandpile model was introduced in 
\cite{peres_levine_strong_spherical}, and it was also used on comb lattices as in \cite{idla-comb-huss-sava} in order to prove the shape result for the IDLA cluster. These two models (IDLA and divisible sandpile) share many properties even though IDLA is a probabilistic growth model based on random walks, while the other one is strictly
deterministic. 

We shortly describe the divisible sandpile model on any graph $G$ here. Suppose we have an initial mass distribution of $n$ particles of sand at the origin $o$ of $G$. At each time-step $k$, we choose a vertex which has mass $\geq 1$. This vertex \emph{topples}, which means that is keeps mass one for itself and distributes the rest equally to the neighbors. For any time $k$, the \emph{odometer function} $u_k:G\to\mathbb{R}$ at vertex $x$ represents the total mass emitted from $x$ during the first $k$ topplings.
Provided
every vertex is chosen infinitely often, it is proven in \cite{peres_levine_strong_spherical}, that as time  $k\to\infty$ the initial mass distribution and the odometer function converge pointwise to a limit distribution $\mu$
and $u$ respectively. The limit mass distribution satisfies $0\leq \mu(x)\leq 1$, for each $x\in G$.
The \emph{divisible sandpile cluster} is defined as the set $\{x\in G: u(x)>0\}$.  The limit functions $\mu$ and $u$ do not depend on the order of topplings.
A precise analysis of the divisible sandpile model on $\SG$ has been done 
 in \cite{sandpile-sg-huss-sava}, where the authors obtain the limit shape for the sandpile cluster.
  
The next result is similar to \cite[Lemma 3]{lawler_bramson_griffeath} and the proof is based on the
divisible sandpile model and the odometer function on $\SG$.

\begin{lem}\label{lem:mean-val-gr-fc}
Fix $\epsilon>0$. For $n$ sufficiently large and $z\in B_o\left(n(1-\epsilon)\right)$, we have
\begin{equation*}
 \sum_{y\in B_o(n)}g_n(y,z)\leq |B_o(n)|g_n(o,z).
\end{equation*}
\begin{proof}
Define the function $h_n:B_o(n)\to \mathbb{R}$, by
\begin{equation*}
 h_n(z)=|B_o(n)|g_n(o,z)- \sum_{y\in B_o(n)}g_n(y,z).
\end{equation*}
By using the linearity of the Laplace operator together with 
\begin{equation*}
 \Delta g_n(o,z)=-\delta_o(z), \ \text{ and } \Delta g_n(y,z)=-\delta_y(z),
\end{equation*}
we obtain that $h_n(z)$ solves the following Dirichlet problem on $\SG$:
\begin{equation*}
 \begin{cases}
   \Delta h_n(z) & =  1-|B_o(n)|\delta_o(z),\quad z\in B_o(n) \\
   h_n(z)        & =  0,\quad z\in B^c_o(n).
 \end{cases}
\end{equation*}
The Laplace of the odometer function $u$ on $\SG$ satisfies equation (3) from
 \cite{sandpile-sg-huss-sava}, when starting with initial sand distribution $\mu_0$. That is, if we start with mass $|B_o(n)|$ at the origin $o$, then the odometer function $u$ satisfies 
 \begin{equation*}
 \begin{cases}
   \Delta u(z) & =  1-|B_o(n)|\delta_o(z),\quad z\in \mathcal{S} \\
   u(z)        & =  0,\quad z\in \mathcal{S}^c,
 \end{cases}
\end{equation*}
where $\mathcal{S}$ is the divisible sandpile cluster on $\SG$. By \cite[Theorem 1.1]{sandpile-sg-huss-sava}, the shape of the sandpile cluster $\mathcal{S}$ when starting with mass
$|B_o(n)|$  at $o$ is given by $B_o(n-1)\subseteq \mathcal{S}\subseteq B_o(n)$.  The difference is there that the balls are closed instead of open, therefore translating \cite[Theorem 1.1]{sandpile-sg-huss-sava} in our setting of open balls, we have $B_o(n)\subseteq \mathcal{S}\subseteq B_o(n+1)$.

Define the function $k_n: B_o(n+1)\to \mathbb{R}$ by $k_n(z):=u(z)-h_n(z)$, which is harmonic  on $B_o(n)$. We have
\begin{equation*}
 \begin{cases}
   \Delta k_n(z) & =  0,\quad z\in B_o(n)\\
    k_n(z)         & =  u(z)\geq 0,\quad z\in \partial B_o(n).
 \end{cases}
\end{equation*}
Applying both the minimum and the maximum principle to $k_n$, if 
$\max_{z\in \partial B_o(n)}k_n(z)=c\geq 0$, then we obtain $u(z)-c<h_n(z)<u(z)$ for all $z\in B_o(n)$. The odometer function $u(z)$ is strictly positive on $\mathcal{S}$ and decreasing in the distance from $z$ to the origin $o$. This means that there exists a constant $c'$, such that $u(z)>c$ on $B_o(n-c')$, which implies that $h_n(z)>0$ on $B_o(n-c')$. For $n$ large enough $ B_o(n(1-\epsilon))\subset B_o(n-c')$, and this yields $h_n(z)>0$ on $B_o(n(1-\epsilon))$ as well, which proves the claim.
\end{proof}
\end{lem}

We  prove next that for every $\epsilon>0$,
\begin{equation}\label{eq:in_bd_equiv}
B_o\left(n(1-\epsilon)\right)\subset \mathcal{I}\left( b_n(1+\epsilon) \right), \text{ for all sufficiently large $n$},
\end{equation}
with probability $1$. Recall that $b_n=|B_o(n)|$. Taking the intersection on both sides over all $\epsilon'<\epsilon$, we get the inner bound $B_o\left(n(1-\epsilon)\right) \subset \mathcal{I}(b_n)$ in Theorem \ref{thm:idla-sierpinski-gasket}.
By Borel-Cantelli, a sufficient condition for \eqref{eq:in_bd_equiv} is
\begin{equation}\label{eq:bc_equiv}
\sum_n \sum_{z\in B_o\left(n(1-\epsilon)\right)}\mathbb{P}[E_z(b_n(1+\epsilon))]<\infty,
\end{equation}
where
\begin{equation*}
E_z(b_n(1+\epsilon)):=\{z\notin \mathcal{I}( b_n(1+\epsilon))\}
\end{equation*}
is the event that $z$ does not belong to the IDLA cluster $\mathcal{I}( b_n(1+\epsilon))$.
We want to show that the probability of this event decreases exponentially in $n$.
Let us first fix $z\in B_o\left(n(1-\epsilon)\right)$ and look at the first $ b_n(1+\epsilon) $ random walks $(X^i(t))_{t\geq 0}$
and $i=1,2,\ldots , b_n(1+\epsilon)$ which build the IDLA cluster. We let these $ b_n(1+\epsilon) $
walks evolve forever, even after they have left the IDLA cluster.
For bounding $\mathbb{P}[E_z(b_n(1+\epsilon))]$ we use the same approach as in \cite{lawler_bramson_griffeath}.
Let $M$ be the number of walks that visit $z$ before exiting the ball $B_o(n)$.
Furthermore, let $L$ be the number of walks that visit $z$ before exiting the ball $B_o(n)$
but after leaving the occupied IDLA cluster.  We have 
\begin{equation*}
E_z(b_n(1+\epsilon)) \subset \{M=L\}.
\end{equation*}
Then for any given $a \geq 0$
\begin{equation}\label{eq:up-bd-ez}
\mathbb{P}[E_z(b_n(1+\epsilon))]<\mathbb{P}[M=L]\leq \mathbb{P}[M\leq a \text{ or } L\geq a]\leq \mathbb{P}[M\leq a]+\mathbb{P}[L\geq a].
\end{equation}
We choose $a$ later so that the above two probabilities can be made sufficiently small. We have to show that $M$ includes more walks on average
while $L$ includes fewer terms. 
Consider the following stopping times:
\begin{align*}
\sigma^i & =\inf\big\{t>0:X^i(t)\notin \mathcal{I}(i-1)\big\}.\\
               & = \text{ the time it takes the $i$-th walk to leave the IDLA cluster}\\
 \tau_n^i & = \inf\{t>0: X^i(t)\notin B_o(n)\}\\
 			    & = \text{ the time it takes the $i$-th walk to leave the ball } B_o(n)\\
 \tau_z^i & = \inf\{t>0: X^i(t)=z\}\\
 				&  = \text{ the time it takes the $i$-th walk to hit $z$}.           
\end{align*}
Recall that all particles start their journey at the fixed origin $o\in\SG$.
In terms of these stopping times we can write
\begin{equation*}
M=\sum_{i=1}^{b_n(1+\epsilon)}\mathbf{1}_{\left\{\tau_z^i<\tau^i_n\right\}} \quad \text{ and } \quad 
\mathbb{E}[M]= b_n(1+\epsilon)\mathbb{P}_o[\tau_z<\tau_n],
\end{equation*}
since the summands in $M$ are iid. On the other hand
\begin{equation*}
L=\sum_{i=1}^{ b_n(1+\epsilon)}\mathbf{1}_{\left\{\sigma^i<\tau_z^i<\tau^i_n\right\}},
\end{equation*}
but the summands in $L$ are not identically distributed and not independent, since after each walk exits the IDLA
cluster, the shape of the cluster is modified. Thus $\mathbb{E}[L]$ is hard to determine, but a good upper
bound for it would suffice. Note that only those walks that exit the IDLA cluster inside the ball $B_o(n)$
contribute to $L$, and for each $y\in B_o(n)$ there is at most one index $i$ for which $X^i(\sigma^i)=y$. Then the walks
started at $y$ that hit $z$ after leaving the ball $B_o(n)$ are independent. So in order to get rid of the dependence of the summands in $L$,
we enlarge the index to all $y\in B_o(n)$, start a random walk at $y$, and look if this walk visits $z$ before leaving $B_o(n)$.
That is, if we let
\begin{equation*}
\tilde{L}=\sum_{y\in B_o(n)}\mathbf{1}^y_{\left\{\tau_z<\tau_n\right\}},
\end{equation*}
where the indicators $\mathbf{1}^y$ correspond to independent random walks starting at $y$, we have that 
$L\leq \tilde{L}$ and 
\begin{equation*}
\mathbb{E}[\tilde{L}]=\sum_{y\in B_o(n)} \mathbb{P}_y[\tau_z<\tau_n].
\end{equation*}
The next step is to use a large deviation result in order to bound the sum of a large number of independent random variables,
but we need to know more about $\mathbb{E}[M]$, $\mathbb{E}[\tilde{L}]$, and the relationship between them. 
Using equation \eqref{eq:st_gr_fc} together with the symmetry of the stopped Green function, we have
\begin{align*}
\mathbb{E}[M] & = \lfloor b_n(1+\epsilon)\rfloor\dfrac{g_n(o,z)}{g_n(z,z)}\\
\mathbb{E}\left[\tilde{L}\right] & =  \dfrac{1}{g_n(z,z)}\sum_{y\in B_o(n)}g_n(y,z)=\dfrac{1}{g_n(z,z)}\sum_{y\in B_o(n)}g_n(z,y)=\dfrac{1}{g_n(z,z)}\mathbb{E}_z[\tau_n]
\end{align*}
Then, by Lemma \ref{lem:mean-val-gr-fc}, we can write
\begin{align}\label{eq:e_m_e_l}
 \mathbb{E}[M]\geq \left(1+\frac{\epsilon}{2}\right)\frac{b_n g_n(o,z)}{g_n(z, z)}\geq
 \left(1+\frac{\epsilon}{2}\right) \dfrac{\sum_{y\in B_o(n)}g_n(y,z)}{g_n(z, z)}=\left(1+\frac{\epsilon}{2}\right) \mathbb{E}[\tilde{L}].
\end{align}

\begin{lem}\label{lem:e-l-bound}
The expectation of the random variable $\tilde{L}$ can be bounded from below by
\begin{equation*}
 \mathbb{E}\left[\tilde{L}\right]\geq c' n^{\alpha},
\end{equation*}
where $c'>0$ depends on nothing but $\epsilon$.
\end{lem}
\begin{proof}
Recall that $\mathbb{E}[\tilde{L}]=\dfrac{1}{g_n(z,z)}\mathbb{E}_z[\tau_n]$. From 
Lemma \ref{lem:exit-time-lower-bound} we know that there exists $c=c(\epsilon)>0$ such that $\mathbb{E}_z[\tau_n(o)]\geq cn^{\beta}$, and Lemma 
\ref{lem:up_bd_gn} gives the bound $g_n(z,z)\leq Cn^{\beta-\alpha}$, for $c$ and $C$ both being positive constants. Putting these two relations together, for $c'=\frac{c}{C}>0 $ we get the claim. 
\end{proof}
On account of \eqref{eq:e_m_e_l}, one has the same lower bound for $\mathbb{E}[M]$.
We shall use the following large deviation estimate for sums of independent indicator random variables. For a proof see \cite[Lemma 4]{lawler_bramson_griffeath}.
\begin{lem}\label{lem:dev-sum-indicators}
Let $S$ be a finite sum of independent indicator random variables with mean $\mu$.
For any $0<\gamma<1/2$, and for all sufficiently large $\mu$,
\begin{equation*}
 \mathbb{P}\left[\left|S-\mu\right|\geq \mu^{1/2+\gamma}\right]\leq 2 \exp\left\{-\frac{1}{4}\mu^{2\gamma}\right\}.
\end{equation*}
\end{lem}
Since both $M$ and $\tilde{L}$ are finite sums of indicator random variables, we can apply the above Lemma to both of them.
Recall that we want to choose a number $a$ such that the probabilities $\mathbb{P}[M\leq a]$ and $\mathbb{P}[\tilde{L}\geq a]$ can be made sufficiently small.

\begin{proof}[Proof of the inner bound in Theorem \ref{thm:idla-sierpinski-gasket}]
Recall that in order to prove that $B(n(1-\epsilon))\subset \mathcal{I}(b_n(1+\epsilon))$, it is sufficient to upper bound the probabilities $\mathbb{P}[M\leq a]$ and $\mathbb{P}[\tilde{L}\geq a]$ and to show that they are summable over $n$ and $z\in B_o\left(n(1-\epsilon)\right)$.
Let us choose $a=\big(1+\frac{\epsilon}{4}\big)\mathbb{E}\left[\tilde{L}\right]$ and $\gamma=\frac{1}{4}$, and let us first show that for $\epsilon >0$ and $n$ large enough we have
\begin{equation}\label{eq:elm-inner-bd}
 \mathbb{E}\left[\tilde{L}\right]+\mathbb{E}\left[\tilde{L}\right]^{1/2+\gamma}\leq \underbrace{\Big(1+\frac{\epsilon}{4}\Big)\mathbb{E}\left[\tilde{L}\right]}_\text{$=a$}\leq
 \mathbb{E}[M]-\mathbb{E}[M]^{1/2+\gamma}.
\end{equation}
The first inequality in \eqref{eq:elm-inner-bd} comes from Lemma \ref{lem:e-l-bound}:
\begin{equation*}
 \mathbb{E}\left[\tilde{L}\right]\Big(1+\mathbb{E}\left[\tilde{L}\right]^{-1/4}\Big)\leq  \mathbb{E}\left[\tilde{L}\right] \Big(1+\frac{1}{c'n^{\alpha/4}}\Big)\leq \Big(1+\frac{\epsilon}{4}\Big
 )\mathbb{E}\left[\tilde{L}\right]=a, \text{ for $n$ large enough}.
\end{equation*}
The second inequality in \eqref{eq:elm-inner-bd} is obtained using equation \eqref{eq:e_m_e_l} and Lemma \ref{lem:e-l-bound}:
\begin{align*}
 \mathbb{E}[M]\Big(1-\mathbb{E}[M]^{-1/4}\Big) & \geq \Big(1+\frac{\epsilon}{2}\Big)\mathbb{E}\left[\tilde{L}\right]\Bigg(1-\dfrac{1}{(1+\frac{\epsilon}{2})^{1/4}}\mathbb{E}\left[\tilde{L}\right]^{-1/4}\Bigg)\\
 &\geq \Big(1+\frac{\epsilon}{2}\Big)\mathbb{E}\left[\tilde{L}\right]\Big(1-\mathbb{E}\left[\tilde{L}\right]^{-1/4}\Big)\\
 & = \Big(1+\frac{\epsilon}{4}\Big)\mathbb{E}\left[\tilde{L}\right]+\frac{\epsilon}{4}\mathbb{E}\left[\tilde{L}\right]\Big(1-\frac{4+2\epsilon}{\epsilon}\mathbb{E}\left[\tilde{L}\right]^{-1/4}\Big)\\
 &\geq \Big(1+\frac{\epsilon}{4}\Big)\mathbb{E}\left[\tilde{L}\right]+\frac{\epsilon}{4}\mathbb{E}\left[\tilde{L}\right]\Big(1-\frac{4+2\epsilon}{\epsilon}\frac{1}{c'n^{\alpha/4}}\Big) \geq \Big(1+\frac{\epsilon}{4}\Big)\mathbb{E}\left[\tilde{L}\right],  \text{ for $n$ large enough}.
\end{align*}
The last inequality follows from the fact that for $n$ sufficiently large the quantity $\Big(1-\frac{4+2\epsilon}{\epsilon}\frac{1}{cn^{\alpha/4}}\Big)$ is greater than $0$ and $\mathbb{E}[\tilde{L}]>0$. Therefore we have proved \eqref{eq:elm-inner-bd}.
Thus, recalling that we defined $a=\big(1+\frac{\epsilon}{4}\big)\mathbb{E}[\tilde{L}]$
and $1/2+\gamma=3/4$
\begin{equation*}
 \mathbb{P}\left[\tilde{L}\geq a\right]\leq \mathbb{P}\Big[\tilde{L}\geq \mathbb{E}[\tilde{L}]+\mathbb{E}[\tilde{L}]^{3/4}\Big]\leq 2\exp\Big\{-\frac{1}{4}\mathbb{E}[\tilde{L}]^{1/2} \Big\}\leq \exp\Big\{-cn^{\frac{\alpha}{2}}\Big\}.
\end{equation*}
where $c=c(\epsilon)>0$ is independent of $n$. The last two inequalities above follow from Lemma \ref{lem:dev-sum-indicators} and Lemma \ref{lem:e-l-bound}. Similarly for $M$, we use Lemma \ref{lem:dev-sum-indicators}, \eqref{eq:e_m_e_l}, and Lemma \ref{lem:e-l-bound} to obtain
\begin{align*}
 \mathbb{P}[M\leq a] & \leq \mathbb{P}\Big[M\leq \mathbb{E}[M]-\mathbb{E}[M]^{3/4}\Big]
 \leq 2\exp\Big\{-\frac{1}{4}\mathbb{E}[M]^{1/2} \Big\}\\
  & \leq 2\exp\Big\{-\frac{1}{4}\left(1+\frac{\epsilon}{2}\right)^{1/2}\mathbb{E}[\tilde{L}]^{1/2} \Big\}\leq 
 2\exp\Big\{-\frac{1}{4}\mathbb{E}[\tilde{L}]^{1/2} \Big\}\leq
 \exp\Big\{-cn^{\frac{\alpha}{2}}\Big\}.
 \end{align*}
Putting together the previous two inequalities and using \eqref{eq:vg}, we have that for $\epsilon>0$ there exists $n_\epsilon\in \mathbb{N}$ such that for all $n\geq n_{\epsilon}$,
\begin{align*}
 \sum_{n\geq n_{\epsilon}}\sum_{z\in B_o\left(n(1-\epsilon)\right)}\mathbb{P}[z\notin \mathcal{I}(b_n(1+\epsilon))]
 & \leq \sum_{n\geq n_{\epsilon}}\sum_{z\in B_o\left(n(1-\epsilon)\right)}\left(\mathbb{P}[M\leq a]+\mathbb{P}[\tilde{L}\geq a]\right)\\
 & \leq  \sum_{n\geq n_{\epsilon}}\sum_{z\in B_o\left(n(1-\epsilon)\right)}2\exp\Big\{-cn^{\frac{\alpha}{2}}\Big\}\\
 & \leq \sum_{n\geq n_{\epsilon}} c n^{\alpha}\exp\big\{-n^{\frac{\alpha}{2}}\big\}<\infty,
 \end{align*}
with $\alpha= \frac{\log{3}}{\log{2}}$. By the Borel-Cantelli Lemma, we have proved that for $\epsilon>0$
\begin{equation*}
 B_o(n(1-\epsilon))\subset \mathcal{I}(b_n(1+\epsilon)), \text{ for $n$ large}
\end{equation*}
with probability $1$, and this implies the inner bound in Theorem \ref{thm:idla-sierpinski-gasket}.
\end{proof}

\subsection{The outer bound for IDLA cluster}
\label{sec:idla-outer-bd}

In order to prove the outer bound $\mathcal{I}(b_n)\subset B_o(n(1+\epsilon))$
in Theorem \ref{thm:idla-sierpinski-gasket}, we shall use parts of the main result of
\cite{idla-outerbd-dum-cop}. More precisely, in well-behaved environments, once a good inner bound is obtained,  we can control the number of particles not contained in the inner bound and obtain a good upper bound, as a Corollary of \cite[Theorem 1.2]{idla-outerbd-dum-cop}. 

In order to be able to use \cite[Theorem 1.2]{idla-outerbd-dum-cop} and
\cite[Corollary 1.3]{idla-outerbd-dum-cop}, one should check that the conditions required in these two results are fulfilled when we perform IDLA on $\SG$. 
Not all the conditions required there are immediately available on $\SG$, but the estimates on the stopped Green function are enough for our purposes. Before going into details about the conditions required, we shall first set the notation about stopped IDLA clusters. We use the same notation as in \cite{idla-outerbd-dum-cop}.

Let $S\subset \SG$ be a finite subset of $\SG$. In order to define the stopped IDLA clusters, we first define the aggregate $\I(S;x)$ when we start with an existing finite cluster $S$, and run a simple random walk  $X(t)$ 
starting at some vertex $x\in S$ until it exits $S$.  
Let $\sigma_S$ be the first time when $X(t)$ exits $S$. Then define
\begin{equation*}
\I(S;x):=S\cup \{X(\sigma_S)\}.
\end{equation*}
For the outer boundary, we will need a slightly more general process, where the growth of the IDLA cluster is stopped before exiting slightly bigger balls. For some radius $r>0$ such that $S\subset B_o(r)$, denote by $\I(S;x\mapsto B_o(r))$ the cluster which is obtained as follows. For some $x\in S$
start a simple random walk $X(t)$ at $x$ and let it run until it either exits $S$ or reaches  $B_o(r)^c$. If $\sigma_S$ is as above, and $\tau_r=\tau_r(o)$ is the first time the random walk exits $B_o(r)$ as defined in \eqref{eq:exit-time}, 
\begin{equation*}
\I\big(S;x\mapsto B_o(r)\big):=S\cup\left\{X(\sigma_S\wedge(\tau_r-1))\right\}.
\end{equation*}
We need to keep track of the paused particles, and their positions on $B_o(r)^c$; the paused particles will all be at distance $r$ from the origin. Define
\begin{equation*}
P(S;x\mapsto B_o(r))=
\begin{cases}
X(\tau_r), & \text{ if } \tau_r\leq \sigma_S\\
\perp,     & \text{otherwise},
\end{cases}
\end{equation*}
where $\perp$ indicates that the random walk attached to the existing aggregate $S$ before exiting the ball $B_o(r)$, so there is no particle to be paused.
For vertices $x_1,x_2,\ldots,x_k$ in $\SG$, a set $S\subset \SG$, and a ball $B_o(r)$, define
$\I\big(S;x_1,x_2,\ldots,x_k \mapsto B_o(r)\big)$ to be the IDLA cluster when starting with the occupied set $S$ and $k$ random walks with starting points $x_1,x_2,\ldots,x_k$, respectively, and paused upon exiting $B_o(r)$. One can then inductively define $\I\big(S;x_1,x_2,\ldots,x_k \mapsto B_o(r)\big)$  by taking 
$$
S_0=S, \quad S_{j}=\I\big(S_{j-1};x_j\mapsto B_o(r)\big)\quad  \text{for} \quad j\in \{1,\ldots,k\},
$$
and $\I\big(S;x_1,x_2,\ldots,x_k \mapsto B_o(r)\big)=S_k$. Since some of the $k$ particles may be stopped on $B_o(r)^c$ before attaching to the existing cluster, we will keep track of these particles in the following way. Define 
$P(S; x_1,\ldots,x_k\mapsto B_o(r))$ to be the sequence of the paused particle in the process above.
More precisely, if $p_j=P(S_{j-1};x_j\mapsto B_o(r))$ for $j\in \{1,\ldots,k\}$, then $P(S;x_1,\ldots,x_k\mapsto B_o(r))$ is just the sequence $(p_j: p_j\neq \perp)$. If particles are not paused before exiting some ball, then the aggregate is simply denoted by $\I(S;x_1,\ldots,x_k)$.
The reason for working with the paused IDLA process, before exiting bigger and bigger balls is because the IDLA process possesses the \textit{abelian property}: the unstopped cluster $\I(S;x_1,\ldots,x_k)$
has the same distribution as 
\begin{equation}\label{eq:idla-abelian-prop}
\I\big(\I(S;x_1,\ldots,x_k\mapsto B_o(r)); P(S;x_1,\ldots,x_k \mapsto B_o(r))\big).
\end{equation}
For details on this property see \cite{diaconis_fulton_1991,lawler_bramson_griffeath}.  
As defined in Section \ref{sec:intro}, the IDLA cluster $\I(n)$ is built by starting $n$ particles at the origin, and letting them run until exiting the previously occupied cluster. We are interested in the shape $\I(b_n)$ of IDLA cluster when we start $b_n=|B_o(n)|$ particles at the origin which is a special case of the stopped process defined above. Nevertheless, in proving the outer bound in Theorem  \ref{thm:idla-sierpinski-gasket}, the general stopped process will be used. In terms of the stopped process defined above, we have
\begin{equation*}
\I(n)=\I\big(\emptyset;\underbrace{o,\ldots,o}_{\text{$n$ times}}\big)
\end{equation*}
and set also
\begin{equation*}
\I_n(x\mapsto r):=\I\big(\emptyset;\underbrace{x,\ldots,x}_{\text{$n$ times}}\mapsto B_x(r)\big)
\end{equation*}
and
\begin{equation*}
P_n(x\mapsto r):=P\big(\emptyset;\underbrace{x,\ldots,x}_{\text{$n$ times}}\mapsto B_x(r)\big).
\end{equation*}
Having set the notations, our aim is to prove a similar result to \cite[Theorem 1.2]{idla-outerbd-dum-cop}. The condition \textbf{weak lower bound} (wLB) from \cite{idla-outerbd-dum-cop}, which says that when releasing $|B_x(n)|$ particles at $x$, the IDLA cluster contains $B_x(n)$ with noticeable probability, does not hold for every starting point $x$. We have such a lower bound on the IDLA cluster in Theorem \ref{thm:idla-sierpinski-gasket}, only when we start $b_n=|B_n(o)|$ particles at the origin $o\in \SG$. We believe that one should be able to adapt the proof of the inner bound for the IDLA cluster when releasing particles at vertices $x\in \SG$ other than $o$, and to get the same result with some additional technical difficulties on Green function estimates and expected exit time from balls for the random walks. Nevertheless, we are not going to do this here. The assumption (wLB) in \cite{idla-outerbd-dum-cop}, which is used only in \cite[Lemma 3.2]{idla-outerbd-dum-cop} will not be needed in our case, because we give a different proof of this Lemma. Our proof uses finer estimates on the stopped Green function, estimates which are not available in the general setting of \cite{idla-outerbd-dum-cop}. Our approach to this result is similar to the one in \cite[Lemma 11]{lawler_1995}. 

The \textbf{continuity assumption } $(C)$ \cite[Page 4/8]{idla-outerbd-dum-cop} 
holds automatically, since the balls we work with are considered with respect to the graph metric, that is, $\rho$ and $d_G$ coincide in our case. Moreover,  \textbf{the regular volume growth} (VG)
condition holds for any radius $r$ and center $x$: $|B_x(r)|$ has growth of order $\alpha$, as defined in \eqref{eq:vg}. The fractal growth $\alpha$ as defined in \eqref{eq:alpha-beta}
will play the role of $d$ in \cite{idla-outerbd-dum-cop}. 
We next prove an estimate for the infimum of the stopped Green function. 

\begin{lem}\label{lem:inf-gr-fr}
Let $u\in (0,1)$. There exists $c=c(u)>0$ such that for all sufficiently large radii $r>0$ and all $x\in \SG$
\begin{equation}\label{eq:greenlower}
\inf_{y\in B_x(ur)}g_r(x,y)\geq c r^{\beta-\alpha},
\end{equation}
where $g_r(x,y)$ is the stopped Green function as defined in \eqref{eq:stopped-green}.
\end{lem}
\begin{proof}
We prove this lemma in two steps. First we show that \eqref{eq:greenlower} holds for $u=\frac{1}{2}$ (and hence for all $u\in \left(0,\frac{1}{2}\right)$). Then we use a chaining argument to extend the estimate to any $u \in \left(\frac{1}{2}, 1\right)$. The elliptic Harnack inequality (EHI) is used in both steps.

\underline{Step 1:} To prove \eqref{eq:greenlower} for $u=\frac{1}{2}$, we need a related Harnack inequality for the stopped Green function stated in \cite{GT02}. We say that the Green function satisfies the condition (HG) if there exist a constant $C_1\geq 1$ such that for all $x\in \SG$, $r>0$, and finite sets $D \supset B_x(r)$,
\begin{align}\tag{HG}
\sup_{y\in B_x(r/2)^c} g_D(x,y) \leq C_1 \inf_{y\in B_x(r/2)} g_D(x,y),
\end{align}
where $g_D(x,y)$ represents the expected number of visits to $y$ before leaving the set $D$, when starting the random walk at $x$.
It is shown in \cite[Theorem 2]{Barlow05} that under the conditions ($p_0$) and (EHI), there exists a constant $C_2\geq 1$ such that if $x_0 \in \SG$, $r>0$, $d(x_0, x)= d(x_0,y) =r/2$, and $B_{x_0}(r) \subset D$, then
\begin{align}
\label{HG}
C_2^{-1} g_D(x_0, y) \leq g_D(x_0, x) \leq C_2 g_D(x_0, y).
\end{align}
In particular this implies that (EHI) and (HG) are equivalent. Based on this equivalence, it can be further shown (see \cite[Lemma 3.7 and Proposition 3.7]{TelcsBook}) that under ($p_0$) and (EHI), there exists $C_3>0$ such that for any ball $B_x(ur)$ with $u \in \left(0, \frac{1}{2}\right]$,
\begin{align}
\label{greenhalf}
\inf_{y\in B_x(ur)} g_r(x,y) \geq C_3 R_{\rm eff}\left(B_x(ur), B_x(r)^c\right).
\end{align}

The conditions ($p_0$) and (EHI) hold on $G$; see Section \ref{sec:prelim-sg}. Furthermore the effective resistance estimate is known on Sierpinski gasket graphs $\SG$; on account of \cite[Proposition 2.3]{TelcsBook} there exist $C \geq c>0$ such that for all $x\in \SG$ and $r>0$,
\begin{align}
\label{res}
c r^{\beta-\alpha} \leq R_{\rm eff}(B_x(r/2), B_x(r)^c) \leq C r^{\beta-\alpha}.
\end{align}
Altogether \eqref{greenhalf} and \eqref{res} yield
\begin{align}
\label{est12}
\inf_{y\in B_x(r/2)} g_r(x,y) \geq c C_3 r^{\beta-\alpha}
\end{align}
for all $r>0$.

\underline{Step 2:} We now extend the estimate to the ball $B_x(ur)$, for $u\in \left(\frac{1}{2},1\right)$. Call $h:=g_r(x,\cdot)$, which is a nonnegative harmonic function on $B_x(r) \setminus \{x\}$. Let $x_1 \in \partial B_x(ur)$ be such that $h(x_1) = \inf_{y\in \partial B_x(ur)} h(y)$.

For $x,y\in \SG$, denote by $\gamma(x,y)$ be the shortest path  in $\SG$ connecting $x$ and $y$. Let $y_1$ be the intersection point of $\gamma(x,x_1)$ and $\partial B_x(r/2-1)$. Along the path $\gamma(y_1, x_1)$ we construct a minimal chain of intersecting balls 
 $\left\{B_{z_k}(r'): k \in \{1,2,\cdots, K\}\right\}$, where $z_k \in \gamma(y_1, x_1)$ for every $k$, $y_1 \in B_{z_1}(r')$, and $x_1 \in B_{z_K}(r')$. In order to apply (EHI) to each ball $B_{z_k}(r') \subset B_{z_k}(2r')$, we choose the radius $r'>0$ such that $B_{z_k}(2r') \subset \left(B_x(r)\setminus\{x\}\right)$ for every $k$. An easy geometric reasoning shows that one can take $r'= \lfloor (1-u)r \rfloor$, and the chain consists of $K$ balls with 
\begin{align}
\label{K}
K \leq \frac{(u-\frac{1}{2})r}{ 2(r'-1) } \leq \frac{1}{2}\frac{u-\frac{1}{2}}{1-u} + \mathcal{O}\left(\frac{1}{r}\right)
\end{align}
as $r\to\infty$. Applying (EHI) to the function $h$ successively yields the comparison
\begin{align}
\label{h1}
h(y_1) \leq C h(z_1) \leq C^{2} h(z_2) \leq \cdots \leq  C^{K+2} h(x_1),
\end{align}
In other words, there exists $C_4 = C_4(u)>0$ such that for all sufficiently large $r$,
\begin{align}
\label{h}
h(x_1) \geq C_4 h(y_1).
\end{align}
Recall the maximum (or minimum) principle for harmonic functions, \emph{e.g.\@} \cite[Proposition 3.1]{TelcsBook}: if $h$ is harmonic on a finite set $A\subset \Gamma$, then
\begin{align}
\label{max}
\max_{\overline A} h= \max_{\partial A}h, \quad
\min_{\overline A} h= \min_{\partial A} h,
\end{align}
where $\overline A = A \cup \partial A$. By taking $A=B_x(ur) \setminus \{x\}$, we have $\partial A = \partial B_x(ur) \cup \{x\}$. By the minimum principle for the function $h:=g_r(x,\cdot)$
on the set $A$ 
\begin{align}
\label{min}
\inf_{A}  h \geq \inf_{\overline{A}} h = \inf_{\partial A} h.
\end{align}
Since $g_r(x,y) \leq g_r(x,x)$ for all $y \in \SG$, $h:=g_r(x,\cdot)$ attains a maximum at $x$,
which together with \eqref{min} implies
$$
\inf_{B_x(ur)} h \geq \inf_{\partial B_x(ur)} h.
$$
Using \eqref{h} and \eqref{est12}, we deduce that for all sufficiently large $r$,
\begin{align}
\inf_{B_x(ur)} h \geq \inf_{\partial B_x(ur)} h=h(x_1) \geq C_4 h(y_1) \geq C_4 \inf_{B_x(r/2)} h \geq C_4 cC_3 r^{\beta-\alpha},
\end{align}
and this proves the claim. 
\end{proof}

We now prove that the probability that a random walk on $\SG$ hits a set whose complement has size at most $\epsilon r^{\alpha}$ is bounded away from zero. The proof is  similar to \cite[Lemma 11]{lawler_1995}. 

\begin{lem}\label{lem:exit-set-ball}
Let $x\in \SG$ and $X(t)$ be a simple random walk on $\SG$ started at $x$. For every $\epsilon>0$, there exists $\eta>0$ such that if $A\subset B_x(r)$,  with $|A| \geq \epsilon |B_x(r)|$, then
$$
\mathbb{P}_x\left[\sigma_A < \tau_r(x)\right]\geq \eta,
$$
where $\sigma_A$ is the first time $X(t)$ hits $A$. 
\end{lem}
\begin{proof}
Without loss of generality we prove the inequality for all sufficiently large $r>0$.
Let $V$ be the number of visits to $A$ of the random walk $\big(X(t)\big)$ started at $x$, and before leaving $B_x(r)$:
\begin{equation*}
V=\sum_{t=0}^{\tau_r(x)-1}\mathbf{1}_{\left\{X(t)\in A\right\}}.
\end{equation*}
Then we have
\begin{equation*}
\mathbb{P}_x[\sigma_A<\tau_r(x)]\geq \mathbb{P}_x[V\geq 1]=\dfrac{\mathbb{E}_x[V]}{\mathbb{E}_x[V|V\geq 1]}.
\end{equation*}
Observe that on the event $\{V\geq 1\}$, $\sigma_A < \tau_r(x)$ and $V\leq \tau_r(x) - \sigma_A$. Combine this with the Markov property and we get
\begin{align}
\nonumber \mathbb{E}_x[V|V\geq 1] &\leq \mathbb{E}_x\left[\tau_r(x)-\sigma_A | V\geq 1\right]\\
\nonumber & = \mathbb{E}_x\left[\mathbb{E}_{X(\sigma_A)}[\tau_r(x)-\sigma_A] | V\geq 1 \right]\\
\nonumber &= \mathbb{E}_x \left[\mathbb{E}_{X(\sigma_A)} [\tau_r(x)]|V\geq 1\right] \\
\label{eq:up-bnd-visits} &\leq \mathbb{E}_x[Cr^{\beta} | V\geq 1] = Cr^\beta,
\end{align}
where in the last line we used Lemma \ref{lem:exit-time-upper-bound}, which states that $\mathbb{E}_y[\tau_r(x)] \leq Cr^\beta$ for all $y\in A\subset B_x(r)$.

Now we need a lower bound of order $\beta$ for $\mathbb{E}_x[V]$. Recall that the balls in $\SG$
have growth of order $\alpha$; see \eqref{eq:vg}. Let $\epsilon>0$. Then we can find $u=u(\epsilon)<1$ such that for all sufficiently large $r>0$
$$
\left|B_x(r) \setminus B_x(ur)\right| \leq  \frac{\epsilon}{2} |B_x(r)|.
$$
Since $\left|A\setminus B_x(ur)\right|\leq \left|B_x(r) \setminus B_x(ur)\right|$ and $A\subset B_x(r)$ with $|A| \geq \epsilon |B_x(r)|$, we have
$$
|A\cap B_x(ur)| =|A|-\left|A\setminus B_x(ur)\right|\geq \frac{\epsilon}{2} |B_x(r)|.
$$ 
On the other hand,
$$
\mathbb{E}_x[V] = \sum_{y\in A} g_r(x,y) \geq \sum_{y\in B_x(ur) \cap A} g_r(x,y) \geq \frac{\epsilon}{2} |B_x(r)| \inf_{y\in B_x(ur)} g_r(x,y).
$$
Lemma \ref{lem:inf-gr-fr} together with equation \eqref{eq:vg} imply that there exists a constant $c=c(u)$ such that
$$
\mathbb{E}_x[V]\geq \frac{\epsilon}{2}c r^{\alpha} r^{\beta-\alpha}=c(\epsilon)r^{\beta},
$$
which together with \eqref{eq:up-bnd-visits} yields
$$
\mathbb{P}_x\left[\sigma_A < \tau_r(x)\right]\geq \frac{c(\epsilon)}{C}=\eta >0,
$$
whence the claim.
\end{proof}

The statement of the next result is similar to \cite[Lemma 3.3]{idla-outerbd-dum-cop}. Nevertheless, the proof there uses the (wLB) on IDLA which we do not have. We use instead Lemma \ref{lem:exit-set-ball}.

\begin{lem}\label{lem:hit-set-ball}
There exist $\rho,\eta \in (0,1]$ such that for large enough $n$ and $n^{\frac{1}{\alpha(\alpha+1)}}<r<n$, the following holds. Let $x\in B_o(n)$ and let $S \subset B_o(n+r)$ be such that $|S\setminus B_o(n)|\leq \rho r^{\alpha}$. Let $X(t)$ be a simple random walk started at $x$, $\sigma_{Q}$ be the first time $(X(t))$ hits the set $Q:=B_o(n+r)\setminus (S\cup B_o(n))$, and $\tau_{n+r}(o)$ be the first time $(X(t))$ exits the ball $B_o(n+r)$ of radius $n+r$ around the origin $o$. Then 
\begin{equation*}
\mathbb{P}_x[\sigma_{Q}<\tau_{n+r}(o)]\geq \eta.
\end{equation*}
\end{lem}

\begin{proof}
For every path $\gamma$ from inside $B_o(n)$ to outside $B_o(n+r)$, let $y(\gamma)$ be the first vertex on this path for which $d\left(y(\gamma),B_o(n)\right)= \frac{r}{2}$. Let us denote by $Y$ the set of all vertices  $y(\gamma)$ for paths $\gamma$. Moreover, every path from $x\in B_o(n)$
to outside $B_o(n+r)$ must hit the set $Y$. Therefore by Markov's property, it suffices to prove the result for starting points $y\in Y$. Let us fix such a $y\in Y$, 
and consider the ball $B_y(r/3)$ of radius $r/3$ around $y$. By letting $A=B_y(r/3)\setminus S$ and using \eqref{eq:vg}, there exists $C\geq 1$ such that
$$
|A|\geq |B_y(r/3)|-|S|\geq \frac{1}{C} \left(\frac{r}{3}\right)^{\alpha}-\rho r^{\alpha}\geq \frac{1}{C}\left(\frac{r}{3}\right)^{\alpha}\Bigg(1-\left(\frac{3}{4}\right)^{\alpha}\Bigg)\geq
\frac{1}{C^2}\Bigg(1-\left(\frac{3}{4}\right)^{\alpha}\Bigg) \cdot \left|B_y(r/3)\right|,
$$
for $\rho=\frac{4^{-\alpha}}{C} \in (0,1]$. Putting $A=B_y(r/3)\setminus S\subset B_y(r/3)$ and $\epsilon=\frac{1}{C^2}\Bigg(1-\left(\frac{3}{4}\right)^{\alpha}\Bigg)$ in Lemma \ref{lem:exit-set-ball},
we then deduce the existence of $\eta>0$ (and, without loss of generality, it is understood that $\eta \leq 1$) such that
$$
\mathbb{P}_y[\sigma_A<  \tau_{r/3}(y)]\geq \eta.
$$
Next, since $d(y,B_o(n))=r/2$, we have that $d(y,\SG\setminus B_o(n+r))>r/3$ and $B_y(r/3)\subset B_o(n+r)-B_o(n)$. Then 
$$
\mathbb{P}_x[\sigma_{Q}<\tau_{n+r}(o)]\geq \mathbb{P}_y[\sigma_A<  \tau_{r/3}(y)]\geq \eta.
$$
\end{proof}

The previous Lemma investigates the behavior of a single particle attaching to the IDLA cluster. The next Lemma, which claims that with high probability, a constant fraction of the IDLA aggregate is absorbed in a fine annulus of $\SG$, resembles \cite[Lemma 3.3]{idla-outerbd-dum-cop}, with $\alpha$ (the volume growth of $\SG$) in place of $d$. 
For the reader's convenience, we state both the result and its proof adapted to our case and to our notation. 

\begin{lem}\label{lem:lemma3.3-duminil}
There exist $\delta>0$ and $p<1$ such that all $n$ large enough, for all $n^{1/(\alpha+1)}<k<n^{\alpha}$
and $x_1,\ldots ,x_k\in B_o(n)$, and for all $S\subset B_o(n)$, the following holds:
$$
\mathbb{P}\left[\left|\mathcal{I}\left(S;x_1,\ldots,x_k \mapsto B_o\left(n+k^{1/\alpha}\right)\right)\setminus S\right|\leq \delta k\right]\leq p^k.
$$
\end{lem}
\begin{proof}
Let $r=k^{1/\alpha}$ and fix $\rho,\eta \in (0,1]$ as in Lemma \ref{lem:hit-set-ball}. Moreover, let $X_1(t),\ldots,X_k(t)$ be simple random walks that start at $x_1,\ldots, x_k$ respectively and stop when exiting $B_o(n+r)$, and that generate the stopped IDLA cluster. Let $k'=\lfloor \rho k\rfloor\leq \rho r^{\alpha}$, and for $j\in\{1,\ldots k'\}$, denote
$$
\mathcal{I}_j=\mathcal{I}\left(S;x_1,\ldots,x_j\mapsto B_o(n+r)\right).
$$
By construction, since only $j$ vertices can add to the IDLA cluster, we have $\left|\mathcal{I}_j\setminus B_o(n)\right|\leq j\leq \rho r^{\alpha}$, so we are in the setting of Lemma \ref{lem:hit-set-ball}, with $\mathcal{I}_j\subset B_o(n+r)$ instead of the set $S$, for all $j\in\{1,\ldots k'\}$. This implies that for all $j\in\{1,\ldots k'\}$
$$
\mathbb{P}[X_{j+1}\cap \left(B_o(n+r)\setminus \mathcal{I}_j\right)\neq \emptyset | \mathcal{I}_j]\geq \eta.
$$
Thus $\left|\mathcal{I}\left(S;x_1,\ldots,x_k\mapsto B_o(n+r)\right)\setminus S\right|$ dominates a $(k',\eta)$-binomial random variable, which implies that there exist $\delta>0$ and $p<1$ depending only on $\rho,\eta$ such that
$$
\mathbb{P}\left[\left|\mathcal{I}\left(S;x_1,\ldots,x_k \mapsto B_o\left(n+r\right)\right)\setminus S\right|\leq \delta k\right]\leq p^k,
$$
and this proves the desired.
\end{proof}
We next show that the condition \textbf{Lower bound} (LB) from \cite{idla-outerbd-dum-cop}
holds on $\SG$. To do this, we first need to estimate the growth of the annulus $B_o(n)\setminus B_o\left(n(1-\epsilon)\right)$ in $\SG$. Recall that $b_n=|B_o(n)|$.

\begin{lem}\label{lem:an-bn} 
For $1/n<\epsilon<1$, the growth of the annulus $B_o(n)\setminus B_o\left(n(1-\epsilon)\right)$ in $\SG$ satisfies the upper bound
\begin{equation}\label{1-epsilon}
b_n-b_{n(1-\epsilon)}\leq 4 \epsilon^{\alpha-1} b_n.
\end{equation}
\end{lem}


\begin{proof}
To motivate the proof, we first carry out the estimate using closed balls, even though the statement calls for open balls. Let $n=2^k$ and $\epsilon=2^{-m}$ for some positive integers $k$ and $m$ with $m<k$. Then $\overline{B}_o(n)$ is the union of two triangles of side $2^k$ joined at $o$. It is easily shown via induction that each triangle has cardinality $(3^{k+1}+3)/2$, so $\mid \overline{B}_o(n)\mid=2[(3^{k+1}+3)/2]-1 = 3^{k+1} +2$. Meanwhile 
one observes that the difference 
$\overline{B}_o(n)\backslash \overline{B}_o(n(1-\epsilon))$ 
consists of two copies of the union of $2^m$ identical triangles each of side 
$2^{k-m}$. Therefore $\mid \overline{B}_o(n) \mid -\mid \overline{B}_o(n(1-\epsilon))\mid$  is less than 
$$
2\cdot 2^m \cdot \frac{3^{k-m+1}+3}{2} = \frac{2^m(3^{k-m+1}+3)}{3^{k+1}+2}\mid \overline{B}_o(n)\mid \leq 2^m (3^{-m} + 3^{-k}) \mid \overline{B}_o(n)\mid \leq 2\left(\frac{2}{3}\right)^m \mid \overline{B}_o(n)\mid = 2 \epsilon^{\alpha-1} \mid\overline{B}_o(n)\mid.
$$

For the actual proof, we consider the case when 
$2^k< n \leq 2^{k+1}$ and 
$ 2^{-m-1}< 
 \epsilon\leq 2^{-m}$ for positive integers $m$ and $k$ with $m<k$.
 Then the difference 
 $B_o(n)\backslash B_o(n(1-\epsilon))$ 
can be covered by a   union of at most $2\cdot 2^{m+1}$ identical triangles each of side 
 $2^{k-m}$. 
 Therefore the left-hand side in 
 \eqref{1-epsilon} is less than 
$$
 { 2\cdot 2^{m+1}\cdot\frac{3^{k-m+1}+3}{2}}   \leq 
 2\frac{2^m(3^{k-m+1}+3)}{3^{k+1}+2} |B_o(n)|
 \leq 
 2^{m+2}3^{-m}|B_o(n)| 
 \leq 4 \epsilon^{\alpha-1} |B_o(n)|.  
$$
\end{proof}

\begin{remark}
If the center of the ball (or annulus) is an arbitrary vertex $x$ of $\SG$ rather than $o$, then a similar argument shows that \eqref{1-epsilon} holds with the constant $4$ replaced by $8$. This is due to the fact that for $2^k <n \leq 2^{k+1}$, the ball $B_x(n)$ can be covered by two 
joint triangles of side $2^{k+1}$.
\end{remark}

\begin{prop}\label{prop:LB}
Condition \textbf{Lower bound} (LB) holds on $\SG$, that is,
$$
\dfrac{\left|\mathcal{I}_{b_n}(o\mapsto n)\right|}{b_n}\to 1, \quad \text{almost surely}.
$$
\end{prop}
\begin{proof}
By construction $\mathcal{I}_{b_n}(o\mapsto n)$ is a subset of $B_o(n)$, that is 
$\dfrac{\left|\mathcal{I}_{b_n}(o\mapsto n)\right|}{b_n}\leq 1$. On the other hand,  from the inner boundary for IDLA cluster in Theorem \ref{thm:idla-sierpinski-gasket} we have that for every $\epsilon>0$, $B_o(n(1-\epsilon))\subset \mathcal{I}(b_n)$, for $n$ large with probability $1$. Actually, the proof of the inner boundary implies the stronger result that $B_o(n(1-\epsilon))\subset\mathcal{I}_{b_n}(o\mapsto n)$, since in the random variables $M$ and $L$ we count only particles that visit a point $z\in B_o(n(1-\epsilon))$ before exiting $B_o(n)$. Therefore, for every $\epsilon>0$, we have
$$
1-\dfrac{b_n-b_{n(1-\epsilon)}}{b_n}=\dfrac{b_{n(1-\epsilon)}}{b_n}\leq \dfrac{\left|\mathcal{I}_{b_n}(o\mapsto n)\right|}{b_n} \quad \text{almost surely}.
$$
On the other hand, Lemma \ref{lem:an-bn} yields
$$
1- 4 \epsilon^{\alpha -1}\leq 1-\dfrac{b_n-b_{n(1-\epsilon)}}{b_n},
$$
and the left hand side goes to $1$ as $\epsilon$ goes to $0$, which together with the first claim of the proof gives $\dfrac{\left|\mathcal{I}_{b_n}(o\mapsto n)\right|}{b_n}\to 1$ almost surely.
\end{proof}

Now we have all ingredients needed for the proof of the outer boundary for IDLA cluster in Theorem \ref{thm:idla-sierpinski-gasket}. This would be an application of \cite[Theorem 1.2 and Corollary 1.3]{idla-outerbd-dum-cop}.
Since there are some minor gaps in their proofs, 
for the sake of completeness, we give the whole proof, adapted to our case, here.
We can prove  \cite[Theorem 1.2]{idla-outerbd-dum-cop} without the \textbf{weak lower bound} (wLB) condition, condition which was used only in \cite[Lemma 3.2]{idla-outerbd-dum-cop}. 

As in \cite[Theorem 1.2]{idla-outerbd-dum-cop}, we construct inductively a sequence of IDLA aggregates $\mathcal{I}_j$, by stopping the particles at different distances $n_j$ from the origin.
If $k_j$ is the number of stopped particles, we choose the next distance $n_{j+1}$,
at which we pause particles again, in terms of $k_j$ and $n_j$. We iterate this procedure until there are fewer than $n_j^{1/(\alpha+1)}$ particles, at which point there are too few particles to affect the limiting outer radius of the IDLA.

\begin{proof}[Proof of the outer bound in Theorem \ref{thm:idla-sierpinski-gasket}]
To prove that for every $\epsilon>0$, $\mathcal{I}(b_n)\subset B_o(n(1+\epsilon))$
for $n$ large enough with probability $1$, we bound the event $[\mathcal{I}(b_n)\not\subset B_o(n(1+\epsilon))]$ by another event whose probability is exponentially decreasing in $n$, and then apply Borel-Cantelli.

As mentioned just above, we define recursively the sequence of aggregates $\mathcal{I}_j$
and the quantities $n_j,P_j,k_j$, $j=0,1,\ldots$ as follows. Fix first $n>0$, and let
\begin{equation*}
\begin{cases}
n_0 & = n\\
\mathcal{I}_0 & =\mathcal{I}_{b_n}(o\mapsto n)\\
P_0 & = P_{b_n}(o\mapsto n)\\
k_0& = |P_0|.
\end{cases}
\end{equation*}
In words, we start the general internal DLA process with $b_n=|B_o(n)|$ particles at $o$ and build the cluster $\mathcal{I}_0$ by stopping particles either when they attach to the existing cluster, or pausing them on  $B_o(n)^c$, that is, when they reach distance $n$ from $o$. So $\mathcal{I}_0$ is a subset of the unstopped IDLA cluster $\mathcal{I}(b_n)$ as defined in the introduction. Then $P_0$ gives the positions of the paused particles, which will continue their journey (only if there are enough particles to contribute to the behavior of the IDLA outer boundary) in order to build the next cluster $\mathcal{I}_1$. If they do not attach to $\mathcal{I}_1$ before reaching the distance $n_1$ (still to be defined) from the root, then they are paused again and used for the subsequent aggregate. Formally, for $j\geq 0$, let
\begin{equation*}
\begin{cases}
n_{j+1}& = 
		\begin{cases}
		n_j+k_j^{1/\alpha} & \text{ if } k_j>n_j^{1/(\alpha+1)}\\
		\infty & \text{ otherwise }.
		\end{cases}
\\
\mathcal{I}_{j+1} & =\mathcal{I}\left(\mathcal{I}_j;P_j\mapsto B_o\left(n_{j+1}\right)\right)\\
P_{j+1} & = P\left(\mathcal{I}_j;P_j\mapsto B_o\left(n_{j+1}\right)\right)\\
k_{j+1}& = |P_{j+1}|.
\end{cases}
\end{equation*}
We continue this iterative construction as long as we have enough particles. Let $J$ be the minimum value of $j$ for which $k_j\leq n_j^{1/(\alpha+1)}$. At this point, the particles left over ($|P_J|$ of them) evolve until attaching to the aggregate, without being paused anymore, and for all $j\geq J+1$, we have $\mathcal{I}_j=\mathcal{I}_{J+1}$. By the abelian property of the internal DLA process \eqref{eq:idla-abelian-prop}, we have that $\mathcal{I}_{J+1}$ and $\mathcal{I}(b_n)$ have the same distribution.
Since the aggregate at time $J$ is stopped before exiting $B_o(n_J)$, it holds $\mathcal{I}_J\subset B_o(n_J)$ and there are exactly $k_J\leq n_J^{1/(\alpha+1)}$ particles left for completing the aggregate $\mathcal{I}_{J+1}$, particles which cannot build too long tentacles. This means that after releasing the last $k_J$ particles, the radius of the ball which contains $\mathcal{I}_{J+1}$
cannot increase with more than $c^*k_J\leq c^*n_J^{1/(\alpha+1)}$, for some $c^*\leq 1$,
which implies that $\mathcal{I}_{J+1}\subset B_o\left(n_J+c^*n_J^{1/(\alpha+1)}\right)$. Then
\begin{equation*}
\mathbb{P}\left[\mathcal{I}(b_n)\not\subset B_o(n(1+\epsilon))\right]\leq \mathbb{P}\left[n_J+c^*n_J^{1/(\alpha+1)}>n(1+\epsilon)\right],
\end{equation*}
and we will upper bound the probability on the right hand side in the previous inequality.
We have $k_J\leq n_J^{1/(\alpha+1)}$ and $ n_{J-1}^{1/(\alpha+1)}< k_{J-1} <\ldots <k_o <n^{\alpha}$.
The fact $k_0 < n^{\alpha}$ follows from $k_0\leq b_n-b_{n(1-\epsilon)}$ together with Lemma \ref{lem:an-bn}, for $n$ large enough.
Moreover, for every $j=1,\ldots, J-1$, we can apply Lemma \ref{lem:lemma3.3-duminil},
by starting with the occupied cluster $\mathcal{I}_{j-1}\subset B_o(n_{j-1})$, and the paused particles  $P_{j-1}\in \partial B_o(n_{j-1})$. The number of paused particles $k_{j-1}$ used to build  $\mathcal{I}_{j}$ fulfill the relation $n_j^{1/(\alpha+1)}< k_{j-1}<n^{\alpha}<n_j^{\alpha}$. Then there exist $\delta<1$ and $p<1$ such that 

\begin{align*}
\mathbb{P}\left[\left|\I_{j}-\I_{j-1}\right|\leq \delta k_{j-1}\right]&=\mathbb{P}[k_j\geq (1-\delta)k_{j-1}]\\
&=\sum_{l=n_j^{1/(\alpha+1)}}^{ n_j^{\alpha}}\mathbb{P}[k_j\geq (1-\delta)k_{j-1}|k_{j-1}=l]\cdot \mathbb{P}[k_{j-1}=l]\\
& = \sum_{l=n_j^{1/(\alpha+1)}}^{ n_j^{\alpha}}\mathbb{P}[k_j\geq (1-\delta)l]\cdot \mathbb{P}[k_{j-1}=l]\\
& \leq \sum_{l=n_j^{1/(\alpha+1)}}^{ n_j^{\alpha}} p^l \cdot \mathbb{P}[k_{j-1}=l]\leq p^{n_j^{1/(\alpha+1)}}\leq p^{n^{1/(\alpha+1)}},\\
\end{align*}
for all $j=1,\ldots , J$. Since $J\leq n^{\alpha}$, together with the union bound we obtain 
\begin{equation}\label{eq:up_bd_kj}
\mathbb{P}[\exists\  1\leq j\leq J:\ k_j\geq (1-\delta)k_{j-1}]\leq \sum_{j=1}^{J}\mathbb{P}[k_j\geq (1-\delta)k_{j-1}]\leq n^{\alpha} p^{n^{1/(\alpha+1)}}.
\end{equation} 
In view of the inclusion of the events $\left\{\forall \ 1\leq j\leq J:\ k_j<(1-\delta)k_{j-1} \right\}\subseteq \left\{k_l<(1-\delta)^lk_0\right\}$ for any fixed $l\leq J$, we get that for any $l\leq J$
\begin{equation*}
\mathbb{P}\left[k_l\geq (1-\delta)^l k_0\right]\leq \mathbb{P}\left[\exists 1\leq j \leq l:\ k_j\geq (1-\delta)k_{j-1}\right].
\end{equation*}
Thus

\begin{align*}
\mathbb{P}\left[\exists 1\leq l\leq J:\ k_l\geq (1-\delta)^lk_0\right] & \leq \mathbb{P}\left[\bigcup_{l\leq J}\{k_l\geq (1-\delta)^lk_0\}\right]\\
& \leq \mathbb{P}\left[\bigcup_{l\leq J}\left\{\exists 1\leq j\leq l:\ k_j\geq (1-\delta)k_{j-1}\right\}\right]\\
& = \mathbb{P}\left[\exists\  1\leq j\leq J:\ k_j\geq (1-\delta)k_{j-1}\right].
\end{align*}
Altogether equation \eqref{eq:up_bd_kj} and the previous inequality imply that for some $\delta<1$ and $p<1$
\begin{equation}\label{eq:bd-main-thm}
\mathbb{P}\left[\exists 1\leq j\leq J:\ k_j\geq (1-\delta)^jk_0\right]\leq 
\mathbb{P}\left[\exists\  1\leq j\leq J:\ k_j\geq (1-\delta)k_{j-1}\right]
\leq n^{\alpha} p^{n^{1/(\alpha+1)}},
\end{equation}
that is
\begin{equation*}
\mathbb{P}\left[\forall \ 1\leq j\leq J: \ k_j < (1-\delta)^j k_0 \right]\geq 1-n^{\alpha} p^{n^{1/(\alpha+1)}}. 
\end{equation*}
In other words, with probability at least $1- n^{\alpha} p^{n^{1/(\alpha+1)}}$
\begin{equation*}
n_J=n+\sum_{j=0}^{J-1}k_j^{1/\alpha}< n+k_0^{1/\alpha}\sum_{j=0}^{J-1}\left((1-\delta)^{1/\alpha}\right)^j<n+k_0^{1/\alpha}\dfrac{1}{1-(1-\delta)^{1/\alpha}}.
\end{equation*}
Meanwhile, if $n_J+c^*n_J^{1/(\alpha+1)}>n(1+\epsilon)$, then there exists $c'$ such that for $n$ big enough
\begin{align*}
& n+k_0^{1/\alpha}\dfrac{1}{1-(1-\delta)^{1/\alpha}}+c'n^{1/(\alpha+1)}> \\
& > n+k_0^{1/\alpha}\dfrac{1}{1-(1-\delta)^{1/\alpha}}+c^*\left(n+k_0^{1/\alpha}\dfrac{1}{1-(1-\delta)^{1/\alpha}}\right)^{\frac{1}{\alpha+1}}>n(1+\epsilon),
\end{align*}
which implies that, for some constant $c_1>0$ and $n$ big enough
\begin{equation*}
k_0^{1/\alpha}>c_1\epsilon n \quad \Rightarrow \quad k_0 > c\epsilon ^\alpha b_n.
\end{equation*}
So by conditioning on the event $n_J<n+k_0^{1/\alpha}\dfrac{1}{1-(1-\delta)^{1/\alpha}}$, we obtain
\begin{align*}
\mathbb{P}&\left[\mathcal{I}(b_n)\not\subset B_o(n(1+\epsilon))\right]\leq \mathbb{P}\left[n_J+c^*n_J^{1/(\alpha+1)}>n(1+\epsilon)\right]\\
& = \mathbb{P}\left[n_J+c^*n_J^{1/(\alpha+1)}>n(1+\epsilon)\bigg|n_J<n+k_0^{1/\alpha}\dfrac{1}{1-(1-\delta)^{1/\alpha}}\right]\cdot \mathbb{P}\left[n_J<n+k_0^{1/\alpha}\dfrac{1}{1-(1-\delta)^{1/\alpha}}\right]\\
& +\mathbb{P}\left[n_J+c^*n_J^{1/(\alpha+1)}>n(1+\epsilon)\bigg|n_J\geq n+k_0^{1/\alpha}\dfrac{1}{1-(1-\delta)^{1/\alpha}}\right]\cdot \mathbb{P}\left[n_J\geq n+k_0^{1/\alpha}\dfrac{1}{1-(1-\delta)^{1/\alpha}}\right]\\
& \leq \mathbb{P}\left[k_0>c\epsilon^{\alpha}b_n\right]+\mathbb{P}\left[n_J\geq n+k_0^{1/\alpha}\dfrac{1}{1-(1-\delta)^{1/\alpha}}\right]\leq \mathbb{P}\left[k_0>c\epsilon^{\alpha}b_n\right]
+ \mathbb{P}\left[\exists 1\leq j\leq J:\ k_j\geq (1-\delta)^jk_0\right].
\end{align*}
Finally, applying Borel-Cantelli to the events involved in the previous inequality, together with the bound in \eqref{eq:bd-main-thm}, and using that $k_0=b_n-|\mathcal{I}_{b_n}(o\mapsto n)|$ gives that for every $\epsilon>0$,
\begin{equation*}
\mathbb{P}\left[\mathcal{I}(b_n)\not\subset B_o(n(1+\epsilon))\quad i.o. \right]\leq \mathbb{P}\left[\dfrac{|\mathcal{I}_{b_n}(o\mapsto n)|}{b_n}<(1-c\epsilon^{\alpha})\quad i.o.\right].
\end{equation*}
By Proposition \ref{prop:LB}, the event on the right-hand side above can happen only finitely many times. Therefore for every $\epsilon>0$,
\begin{equation*}
\mathcal{I}(b_n)\subset B_o(n(1+\epsilon))\quad \text{ for all sufficiently large } n
\end{equation*}
with probability $1$. This concludes the proof of the outer bound in Theorem \ref{thm:idla-sierpinski-gasket}.
\end{proof}

\section*{Questions}


\paragraph{1. Rotor-router aggregation on Sierpinski gasket graphs.}
\emph{Rotor-router aggregation} is a deterministic version of IDLA, which describes the growth of a cluster of particles, where the particles perform deterministic walks (called \emph{rotor-router walks}) instead of random walks. In a rotor-router walk on a graph $G$, each vertex is equipped with an arrow (rotor) pointing to one of the neighbors. A particle performing a
rotor-router walk first changes the rotor at the current position to point to the next neighbor, in a fixed order of neighbors chosen at the beginning, and then the particle moves to the neighbor the rotor is pointing towards. In rotor-router aggregation, for a fixed initial configuration of rotors, we start $n$ particles at the origin of $o$, and let each of these particles perform rotor-router walk until reaching a site previously unvisited, where it stops. Then a new particle starts at the origin, without resetting the configuration of rotors. The resulting deterministic set $\mathcal{R}(n)$ of $n$ occupied sites is called the \emph{rotor-router cluster}. As in the case of IDLA, one of the questions here is to determine if the set $\mathcal{R}_n$ of occupied sites has a limiting shape regardless of the initial configuration of rotors. IDLA and rotor-router aggregation have similar behavior on several state spaces, as shown  in \cite{peres_levine_strong_spherical} on $\mathbb{Z}^d$, and in \cite{huss-sava-rr-comb,idla-comb-huss-sava} on comb lattices. On the Sierpinski gasket $\SG$,
it has been proven in \cite{chen-kudler-flam} that the rotor cluster has the same limit shape from Theorem \ref{thm:idla-sierpinski-gasket}. Even more, a fourth growth model, called \emph{abelian sandpile} has the same limit shape on $\SG$; see again \cite{chen-kudler-flam}.

\paragraph{2. IDLA on other fractal graphs.}
Another fractal graph with interesting properties is the graphical Sierpinski carpet $\SC$, where
it may be very interesting to investigate the IDLA process. Most of the computations in the current paper can be carried over to Sierpinski carpet graphs, with the exception of Lemma \ref{lem:mean-val-gr-fc} which is more delicate on $\SC$. The proof of Lemma \ref{lem:mean-val-gr-fc} uses the divisible sandpile model on $\SG$, whereon explicit computations can be carried out thanks to the finite ramification of $\SG$. In contrast, $\SC$ is infinitely ramified, and fine estimates of the solution to the corresponding Dirichlet problem are not known at the moment.
According to computer simulations, there does not seem to exist a unique scaling limit for the IDLA clusters. 
Actually, the simulations suggest that there is a whole family of scaling limits,
and that these scaling limits seem to have a fractal boundary.
Figure \ref{fig:idla_sc} shows IDLA clusters on the graphical Sierpinski carpet $\SC$ in dimension $2$,
for $10000$ up to $150000$ random walks starting at the origin.

\begin{figure}
\centering
\begin{tikzpicture}
\node (a1)          at (0,0) {\scalebox{1}[-1]{\includegraphics[width=3.1cm]{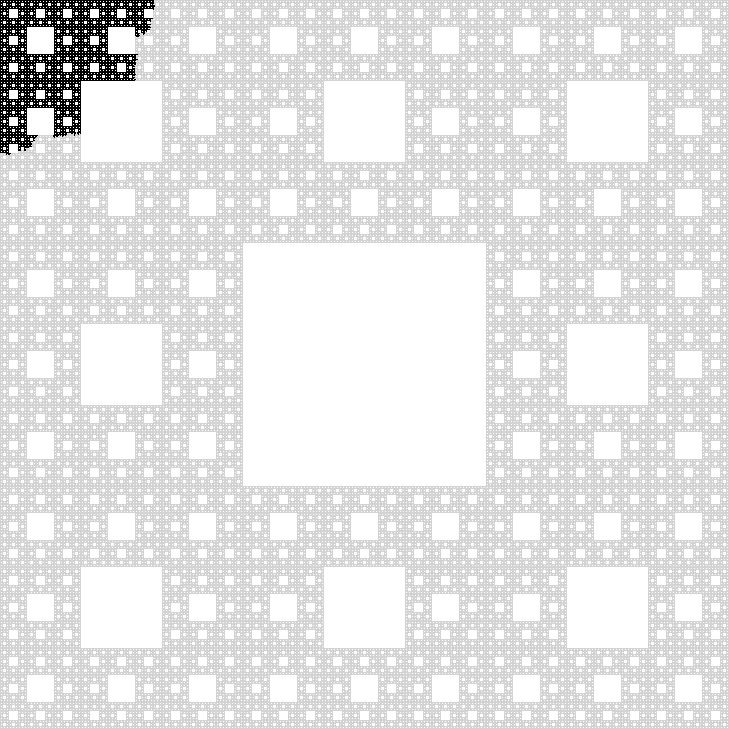}}};
\node (b1) [right=0.5cm of a1] {\scalebox{1}[-1]{\includegraphics[width=3.1cm]{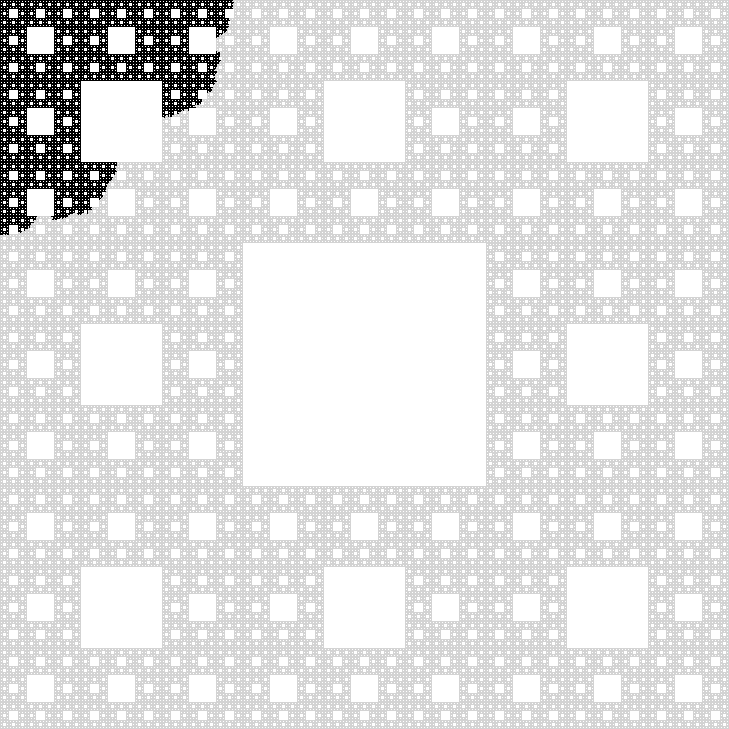}}};
\node (c1) [right=0.5cm of b1] {\scalebox{1}[-1]{\includegraphics[width=3.1cm]{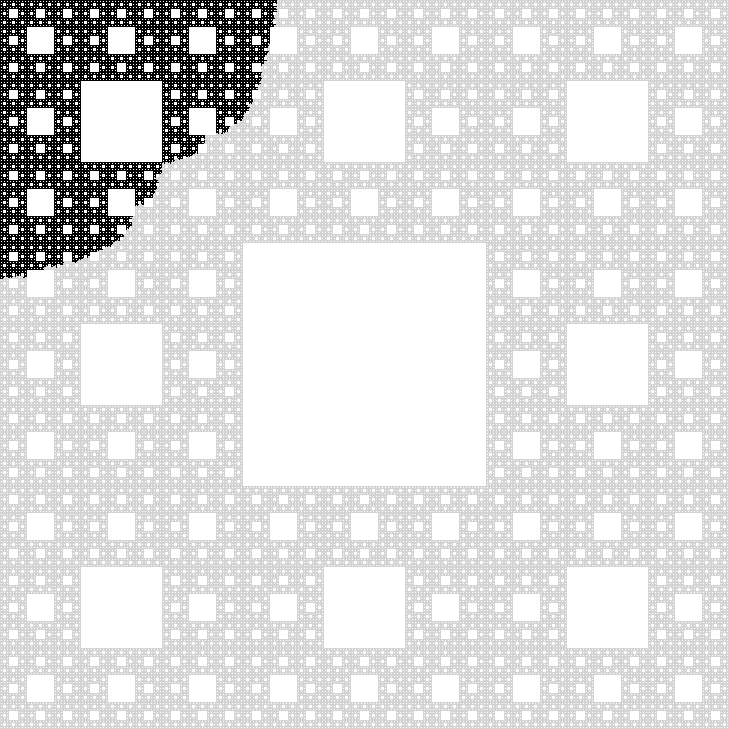}}};
\node (d1) [right=0.5cm of c1] {\scalebox{1}[-1]{\includegraphics[width=3.1cm]{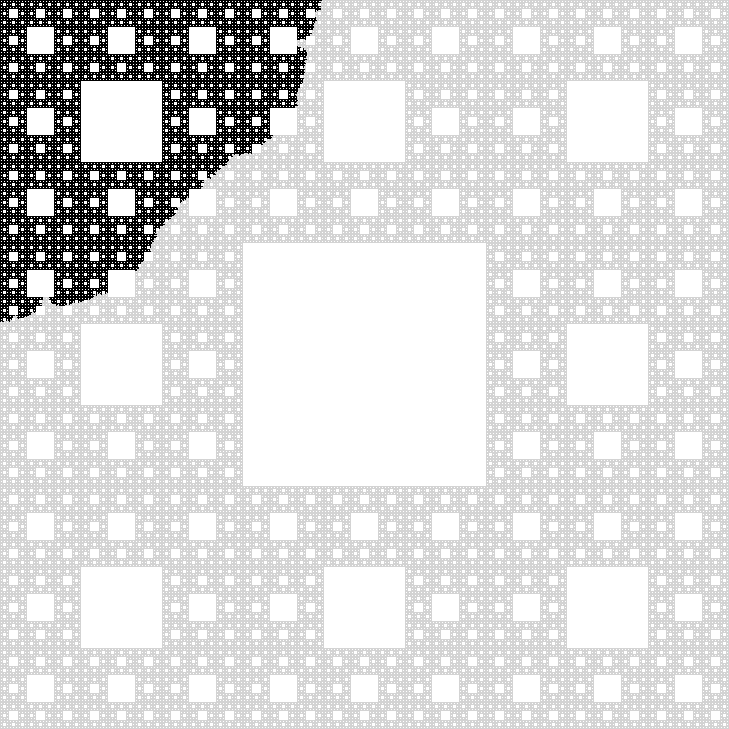}}};
\node (a2) [below=0.5cm of a1] {\scalebox{1}[-1]{\includegraphics[width=3.1cm]{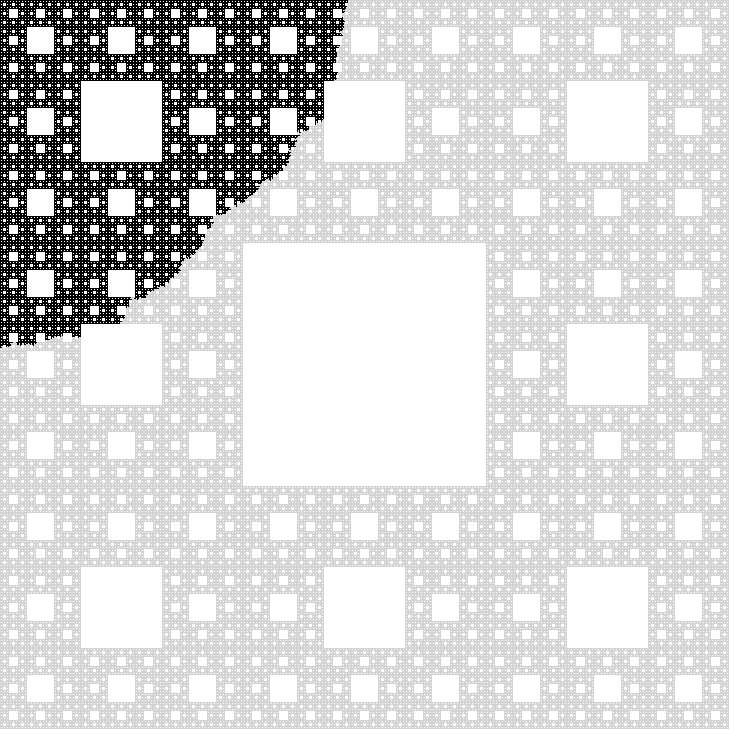}}};
\node (b2) [right=0.5cm of a2] {\scalebox{1}[-1]{\includegraphics[width=3.1cm]{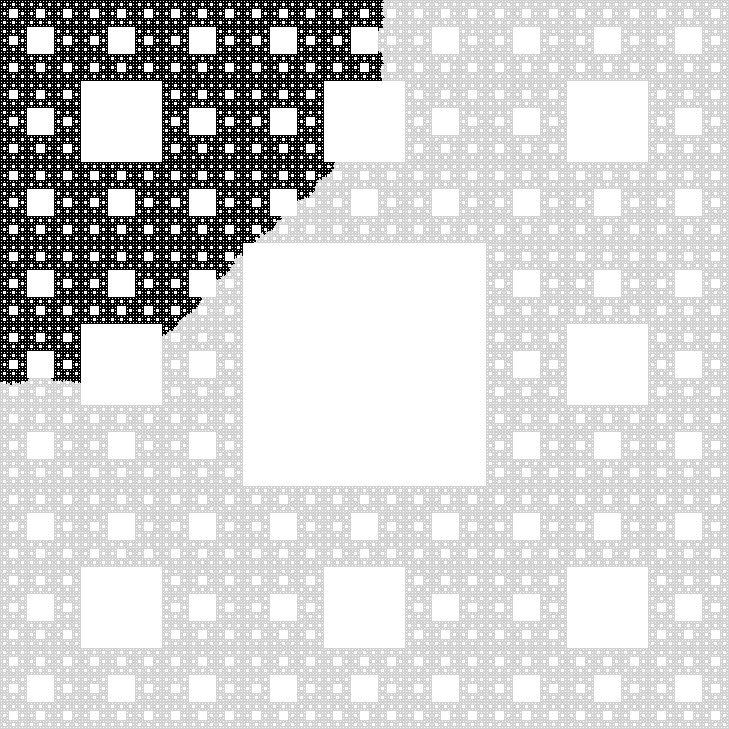}}};
\node (c2) [right=0.5cm of b2] {\scalebox{1}[-1]{\includegraphics[width=3.1cm]{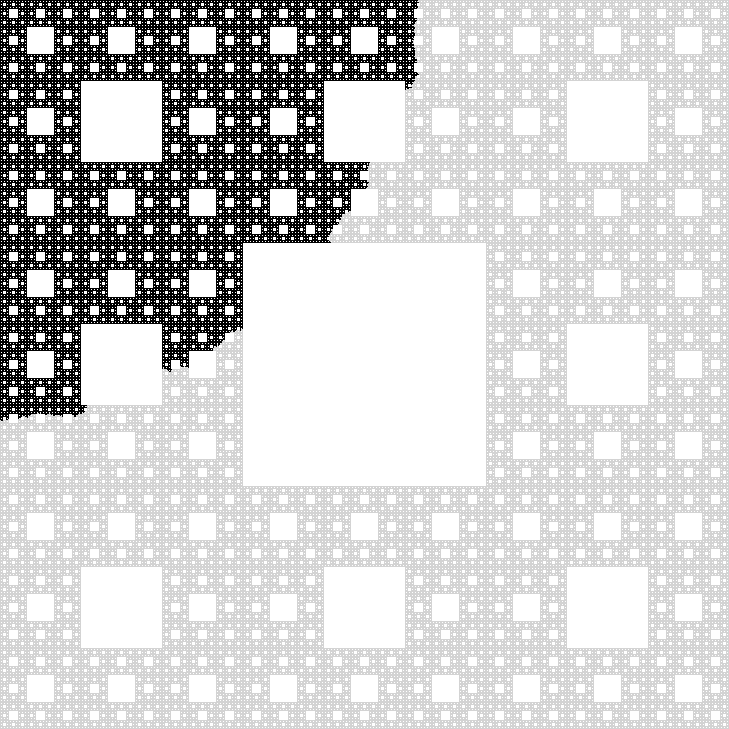}}};
\node (d2) [right=0.5cm of c2] {\scalebox{1}[-1]{\includegraphics[width=3.1cm]{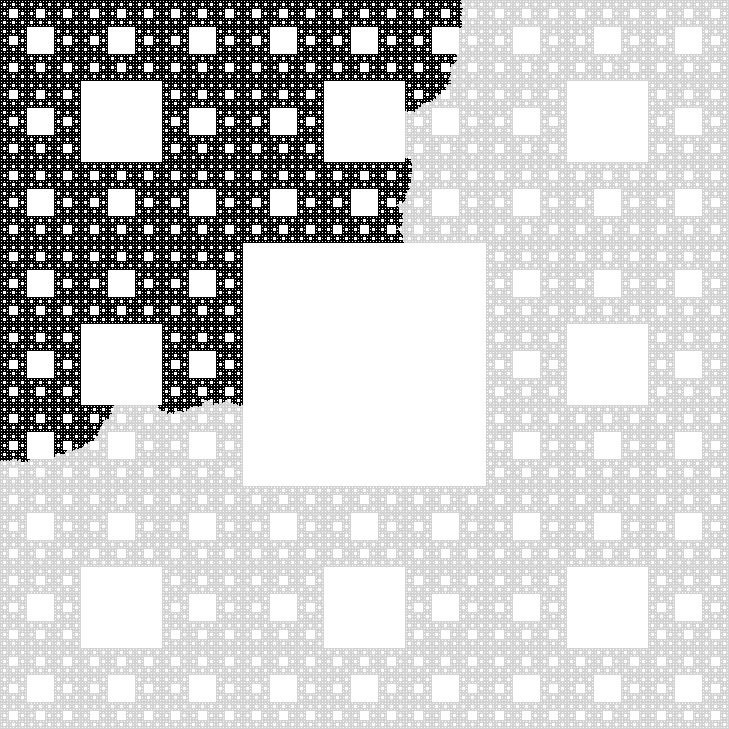}}};
\node (a3) [below=0.5cm of a2] {\scalebox{1}[-1]{\includegraphics[width=3.1cm]{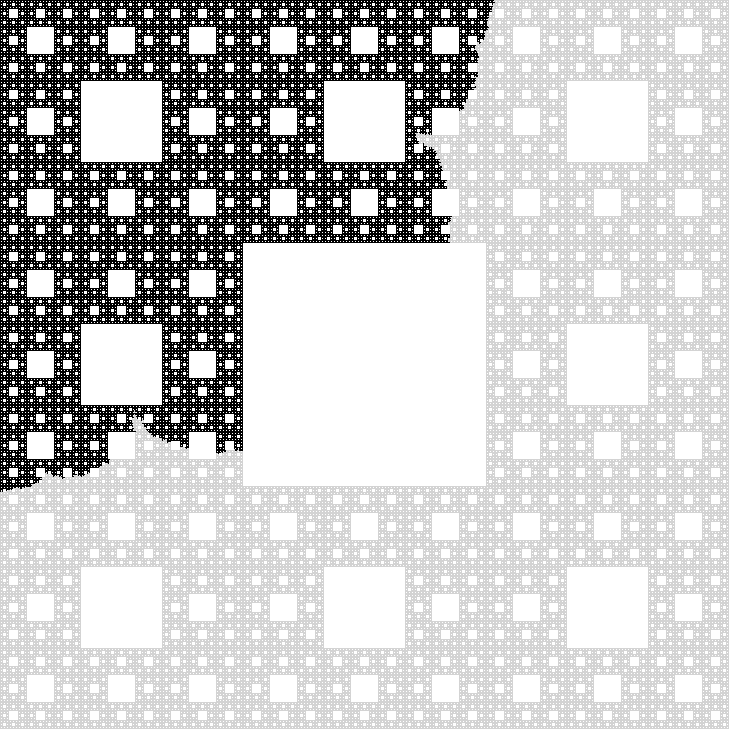}}};
\node (b3) [right=0.5cm of a3] {\scalebox{1}[-1]{\includegraphics[width=3.1cm]{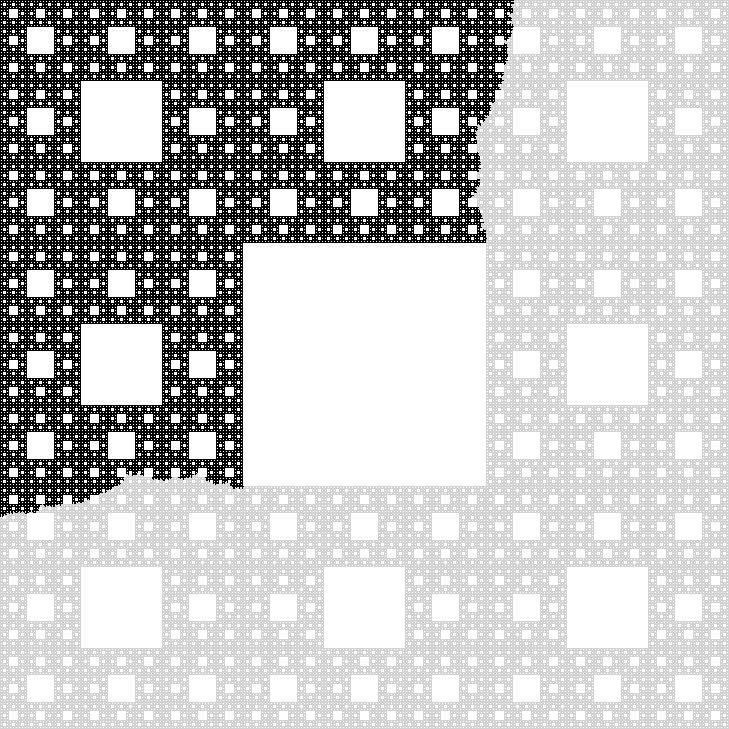}}};
\node (c3) [right=0.5cm of b3] {\scalebox{1}[-1]{\includegraphics[width=3.1cm]{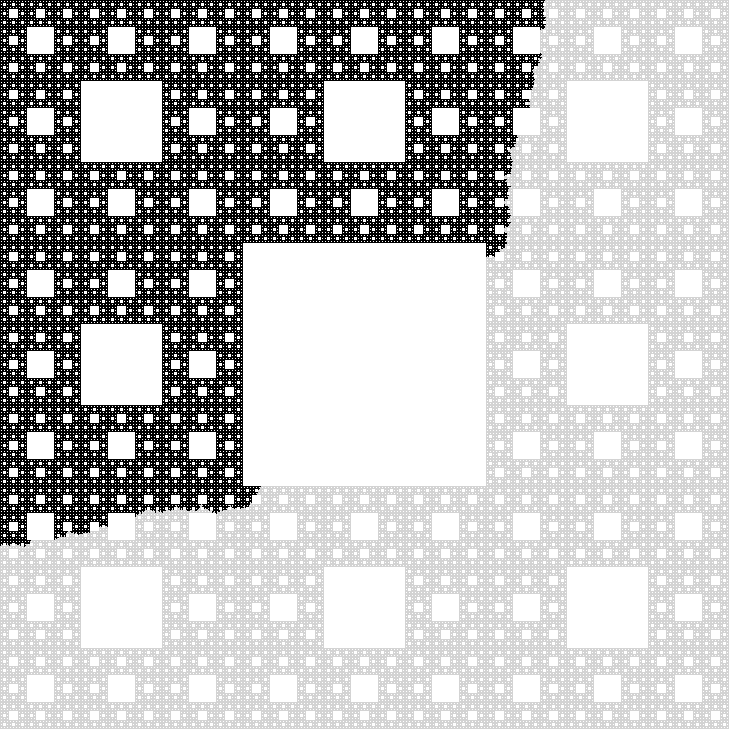}}};
\node (d3) [right=0.5cm of c3] {\scalebox{1}[-1]{\includegraphics[width=3.1cm]{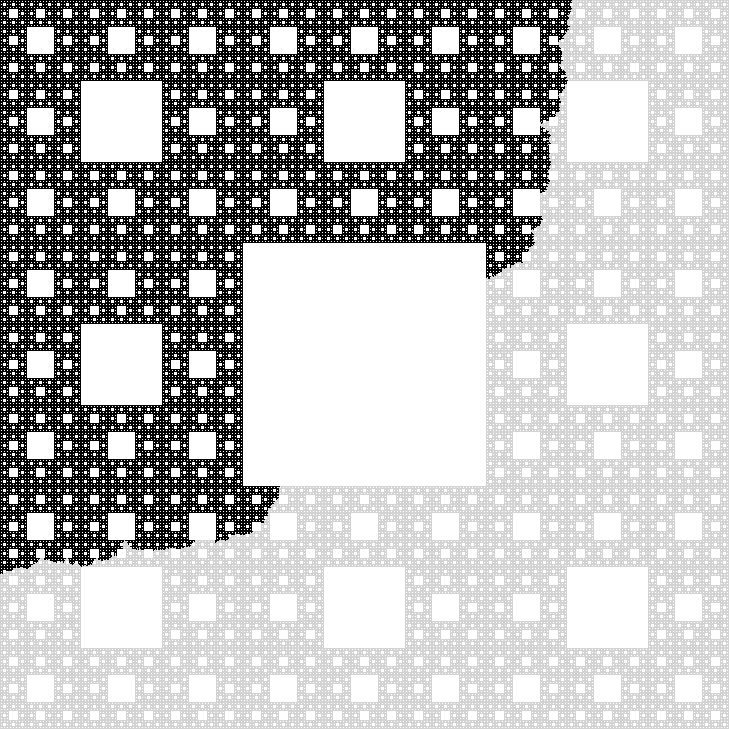}}};
\node (a4) [below=0.5cm of a3] {\scalebox{1}[-1]{\includegraphics[width=3.1cm]{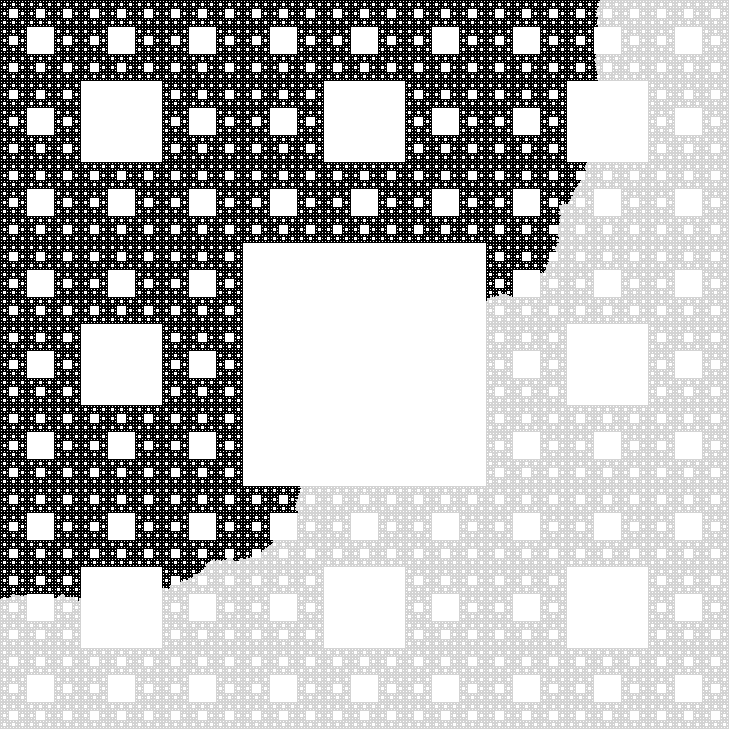}}};
\node (b4) [right=0.5cm of a4] {\scalebox{1}[-1]{\includegraphics[width=3.1cm]{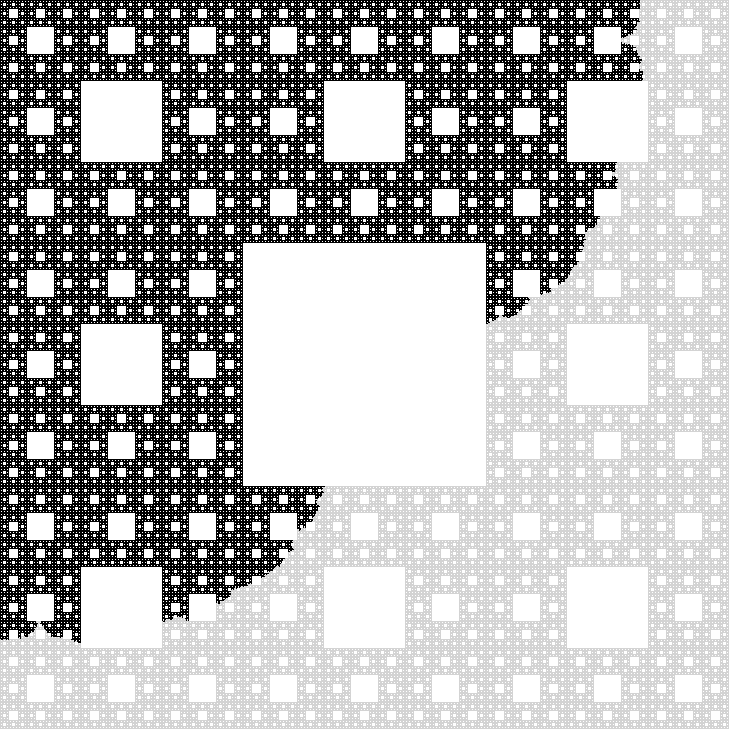}}};
\node (c4) [right=0.5cm of b4] {\scalebox{1}[-1]{\includegraphics[width=3.1cm]{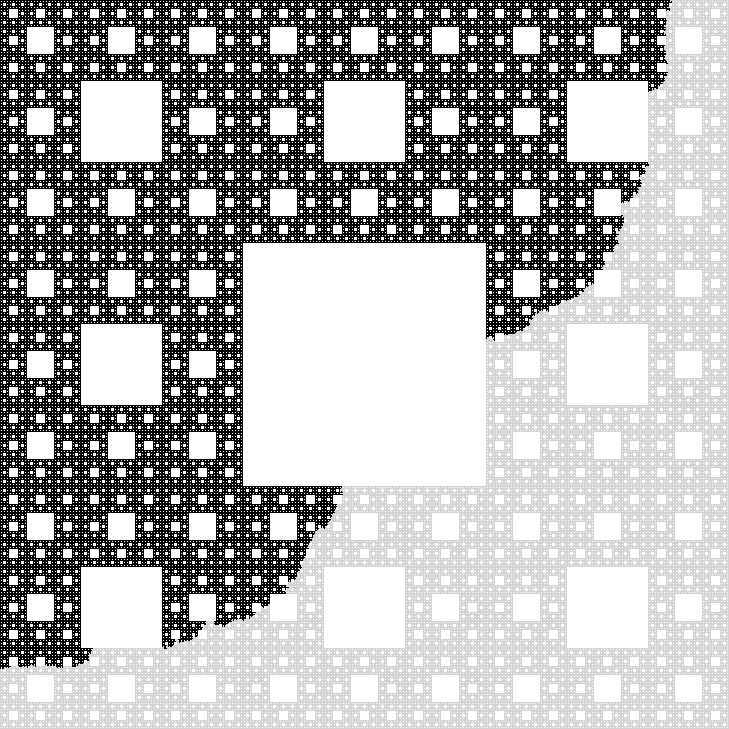}}};
\end{tikzpicture}
\caption{\label{fig:idla_sc}IDLA clusters on the Sierpinski carpet for $10000$ up to $150000$ particles.}
\end{figure}

\paragraph{Acknowledgements.}
We are grateful to Lionel Levine for providing useful comments on an earlier version of this paper.
The research of Joe P.\@ Chen was supported by the National Science Foundation (NSF) grant DMS-1262929, the Simons Foundation grant \#523544, and the Research Council of Colgate University.
The research of Wilfried Huss was supported by the Austrian Science Fund (FWF): 
\href{http://pf.fwf.ac.at/en/research-in-practice/in-the-spotlight-schroedinger/list-of-schroedinger-fellows/2014/8523}{J3628-N26} and P25510-N26.
The research of Ecaterina Sava-Huss was supported by the Austrian Science Fund (FWF):
\href{http://pf.fwf.ac.at/en/research-in-practice/in-the-spotlight-schroedinger/list-of-schroedinger-fellows/2014/210946}{J3575-N26}.
The research of Alexander Teplyaev was supported by the National Science Foundation (NSF) grants DMS-1262929 and DMS-1613025.

\pagestyle{empty}
\addcontentsline{toc}{section}{References}
\bibliography{idla-gasket}

\begin{thebibliography}{DCLYY13}

\bibitem[AR16]{asselah-comb}
A.~Asselah and H.~Rahmani, \textsl{ Fluctuations for internal {DLA} on the
  comb},
\newblock Ann. Inst. Henri Poincar\'e Probab. Stat. \textbf{ 52}(1), 58--83
  (2016).

\bibitem[Bar98]{Bar98}
M.~T. Barlow,
\newblock Diffusions on fractals,
\newblock in \textsl{ Lectures on probability theory and statistics
  ({S}aint-{F}lour, 1995)}, volume 1690 of \textsl{ Lecture Notes in Math.},
  pages 1--121, Springer, Berlin, 1998.

\bibitem[Bar03]{Barlow-heat-kernels-03}
M.~T. Barlow,
\newblock Heat kernels and sets with fractal structure,
\newblock in \textsl{ Heat kernels and analysis on manifolds, graphs, and
  metric spaces ({P}aris, 2002)}, volume 338 of \textsl{ Contemp. Math.}, pages
  11--40, Amer. Math. Soc., Providence, RI, 2003.

\bibitem[Bar05]{Barlow05}
M.~T. Barlow, \textsl{ Some remarks on the elliptic Harnack inequality},
\newblock Bull. London Math. Soc. \textbf{ 37}(2), 200--208 (2005).

\bibitem[BB07]{blachere-brofferio-idla}
S.~Blach{\`e}re and S.~Brofferio, \textsl{ Internal diffusion limited
  aggregation on discrete groups having exponential growth},
\newblock Probab. Theory Related Fields \textbf{ 137}(3-4), 323--343 (2007).

\bibitem[BCK05]{barlow-coulhon-kumagai-hk-2005}
M.~T. Barlow, T.~Coulhon and T.~Kumagai, \textsl{ Characterization of
  sub-{G}aussian heat kernel estimates on strongly recurrent graphs},
\newblock Comm. Pure Appl. Math. \textbf{ 58}(12), 1642--1677 (2005).

\bibitem[BP88]{BP88}
M.~T. Barlow and E.~A. Perkins, \textsl{ Brownian motion on the {S}ierpi\'nski
  gasket},
\newblock Probab. Theory Related Fields \textbf{ 79}(4), 543--623 (1988).

\bibitem[CKF20]{chen-kudler-flam}
J.~P. Chen and J.~Kudler-Flam, \textsl{ Laplacian growth $\&$ sandpiles on the
  {S}ierpinski gasket: limit shape universality and exact solutions},
\newblock Ann. Inst. Henri Poincar\'e (D), to appear  (2020),
\newblock Preprint available at \url{https://arxiv.org/abs/1807.08748}.

\bibitem[DCLYY13]{idla-outerbd-dum-cop}
H.~Duminil-Copin, C.~Lucas, A.~Yadin and A.~Yehudayoff, \textsl{ Containing
  internal diffusion limited aggregation},
\newblock Electron. Commun. Probab. \textbf{ 18}, no. 50, 8 (2013).

\bibitem[DF91]{diaconis_fulton_1991}
P.~Diaconis and W.~Fulton, \textsl{ A growth model, a game, an algebra,
  {L}agrange inversion, and characteristic classes},
\newblock Rend. Sem. Mat. Univ. Politec. Torino \textbf{ 49}(1), 95--119 (1993)
  (1991),
\newblock Commutative algebra and algebraic geometry, II (Italian) (Turin,
  1990).

\bibitem[GKQS14]{GKQS14}
Z.~Guo, R.~Kogan, H.~Qiu and R.~S. Strichartz, \textsl{ Boundary value problems
  for a family of domains in the {S}ierpinski gasket},
\newblock Illinois J. Math. \textbf{ 58}(2), 497--519 (2014).

\bibitem[GT02]{GT02}
A.~Grigor'yan and A.~Telcs, \textsl{ Harnack inequalities and sub-Gaussian
  estimates for random walks},
\newblock Math. Ann. \textbf{ 324}(3), 521--556 (2002).

\bibitem[HS11]{huss-sava-rr-comb}
W.~Huss and E.~Sava, \textsl{ Rotor-router aggregation on the comb},
\newblock Electron. J. Combin. \textbf{ 18}(1), Paper 224, 23 (2011).

\bibitem[HS12]{idla-comb-huss-sava}
W.~Huss and E.~Sava, \textsl{ Internal aggregation models on comb lattices},
\newblock Electron. J. Probab. \textbf{ 17}, no. 30, 21 (2012).

\bibitem[HSH19]{sandpile-sg-huss-sava}
W.~Huss and E.~Sava-Huss, \textsl{ Divisible sandpile on Sierpinski gasket
  graphs},
\newblock Fractals \textbf{ 27}(03), 1950032 (2019).

\bibitem[Hus08]{huss-nonamenable}
W.~Huss, \textsl{ Internal diffusion-limited aggregation on non-amenable
  graphs},
\newblock Electron. Commun. Probab. \textbf{ 13}, 272--279 (2008).

\bibitem[Kig01]{KigamiBook}
J.~Kigami,
\newblock \textsl{ Analysis on fractals}, volume 143 of \textsl{ Cambridge
  Tracts in Mathematics},
\newblock Cambridge University Press, Cambridge, 2001.

\bibitem[KL12]{keller-lenz}
M.~Keller and D.~Lenz, \textsl{ Dirichlet forms and stochastic completeness of
  graphs and subgraphs},
\newblock J. Reine Angew. Math. \textbf{ 666}, 189--223 (2012).

\bibitem[Law95]{lawler_1995}
G.~Lawler, \textsl{ Subdiffusive fluctuations for internal diffusion limited
  aggregation},
\newblock Ann. Probab. \textbf{ 23}, 71--86 (1995).

\bibitem[LBG92]{lawler_bramson_griffeath}
G.~F. Lawler, M.~Bramson and D.~Griffeath, \textsl{ Internal diffusion limited
  aggregation},
\newblock Ann. Probab. \textbf{ 20}(4), 2117--2140 (1992).

\bibitem[LP09]{peres_levine_strong_spherical}
L.~Levine and Y.~Peres, \textsl{ Strong spherical asymptotics for rotor-router
  aggregation and the divisible sandpile},
\newblock Potential Anal. \textbf{ 30}(1), 1--27 (2009).

\bibitem[Luc14]{lucas-idla-drift}
C.~Lucas, \textsl{ The limiting shape for drifted internal diffusion limited
  aggregation is a true heat ball},
\newblock Probab. Theory Related Fields \textbf{ 159}(1-2), 197--235 (2014).

\bibitem[SH20]{survey-dla}
E.~Sava-Huss, \textsl{ From fractals in external DLA to internal DLA on
  fractals},
\newblock Fractal Geometry \& Stochastics VI, to appear  (2020),
\newblock Preprint available at \url{https://arxiv.org/abs/1902.03800}.

\bibitem[She10]{shellef-percolation}
E.~Shellef, \textsl{ I{DLA} on the supercritical percolation cluster},
\newblock Electron. J. Probab. \textbf{ 15}, no. 24, 723--740 (2010).

\bibitem[Str06]{StrichartzBook}
R.~S. Strichartz,
\newblock \textsl{ Differential equations on fractals: A tutorial},
\newblock Princeton University Press, Princeton, NJ, 2006.

\bibitem[SU00]{splines}
R.~S. Strichartz and M.~Usher, \textsl{ Splines on fractals},
\newblock Math. Proc. Cambridge Philos. Soc. \textbf{ 129}(2), 331--360 (2000).

\bibitem[Tel06]{TelcsBook}
A.~Telcs,
\newblock \textsl{ The art of random walks}, volume 1885 of \textsl{ Lecture
  Notes in Mathematics},
\newblock Springer-Verlag, Berlin, 2006.

\bibitem[WS83]{dla}
T.~A. Witten and L.~M. Sander, \textsl{ Diffusion-limited aggregation},
\newblock Phys. Rev. B \textbf{ 27}, 5686--5697 (May 1983).

\end{thebibliography}
\bibliographystyle{myhep}

\textsc{Joe P. Chen, Department of Mathematics, Colgate University, Hamilton, NY 13346, USA.}\\
\texttt{jpchen@colgate.edu} 

\textsc{Wilfried Huss, ADB Safegate Austria, 8020 Graz, Austria.}\\
\texttt{husswilfried@gmail.com} 

\textsc{Ecaterina Sava-Huss, Department of Mathematics, University of Innsbruck, 6020 Innsbruck, Austria.}
\texttt{Ecaterina.Sava-Huss@uibk.ac.at}

\textsc{Alexander Teplyaev, Department of Mathematics, University of Connecticut, Storrs, CT 06269, USA.}
\texttt{teplyaev@uconn.edu}
\end{document}